\documentclass[12pt,a4paper,reqno]{amsart}

\usepackage[applemac]{inputenc}
\usepackage{amsmath}
\usepackage{amsthm}
\usepackage{amscd}
\usepackage{amsfonts}
\usepackage{amssymb}
\usepackage{mathrsfs}
\usepackage{geometry}
\usepackage{enumerate}
\usepackage{xcolor}
\usepackage{ulem}
\newgeometry{top=3cm,bottom=2.5cm,left=2.5cm,right=2.5cm}

\newtheorem{thm}{Theorem}[section]
\newtheorem{cor}[thm]{Corollary}
\newtheorem{lem}[thm]{Lemma}
\newtheorem{prop}[thm]{Proposition}

\numberwithin{equation}{subsection}

\theoremstyle{definition}
\newtheorem{defn}[thm]{Definition}
\newtheorem{rem}[thm]{Remark}

\newtheorem{exam}[thm]{Example}

\newcommand{\fsG}{\mathscr{G}}

\newcommand{\US}{\mathfrak{A}}

\newcommand{\fcB}{\mathcal{B}}

\newcommand{\frb}{\mathfrak{b}}
\newcommand{\frh}{\mathfrak{h}}

\newcommand{\frm}{\mathfrak{m}}

\def\f#1{\mathfrak{#1}}
\def\fc#1{\mathcal{#1}}
\newcommand{\fS}{\f{S}}

\def\MN(#1){ M_{#1}(\mathbb{N})}
\def\MNR(#1,#2){ M_{#1}(\mathbb{N})_{#2}}
\def\MNS(#1){ M_{#1}(\mathbb{N})^{\pm}}

\newcommand{\ZZ}{\mathbb{Z}}
\newcommand{\QQ}{\mathbb{Q}}
\newcommand{\ZG}{{{\mathbb{Z}}}_2}
\newcommand{\NN}{\mathbb{N}}

\def\MZ(#1){ M_{#1}(\ZG)}
\def\NZ(#1){ (\NN|\ZG)^{#1}}
\def\NZST(#1,#2,#3){ (\NN|\ZG)^{#1}_{#2|#3}}
\def\NZS(#1,#2){ {\NN}^{#1}_{#2}}
\def\MNZ(#1,#2){ M_{#1}(\NN | \ZG)_{#2}}
\def\MNZN(#1){ M_{#1}(\NN | \ZG)}
\def\CMNZ(#1,#2,#3){\Lambda(#1,#2|#3)}
\def\CMN(#1,#2){\Lambda(#1,#2)}
\def\CMNP(#1,#2){\Lambda_{#1,#2}}
\def\MNZNS(#1){ \MNZN(#1)^{\pm}}
\def\SE#1{{#1}^{\bar{0}}}
\def\SO#1{{#1}^{\bar{1}}}
\def\SEE#1{{#1}^{\bar{0}}}
\def\SOE#1{{#1}^{\bar{1}}}

\def\SS(#1,#2){{#1}^{\ol{#2}}}
\def\SSE(#1,#2){{#1}^{\ol{#2}}}

\def\ESE(#1, #2, #3){ \SEE{#1}_{{#2},{#3}} }
\def\ESO(#1, #2, #3){ \SOE{#1}_{{#2},{#3}} }
\def\bs#1{\boldsymbol{#1}}

\def\Qqs(#1,#2){ \mathcal{Q}(#1,#2) }

\def\lcase#1{\MakeLowercase{#1}}

\def\ol#1{\overline{#1}}

\def\wave#1{\widetilde{#1}}
\newcommand{\ep}{\epsilon}
\newcommand{\Qv}{\mathbb Q({v})}

\def\qn{\mathfrak{\lcase{q}}_n}
\def\qne{\mathfrak{\lcase{q}}_n^{\ol{0}}}
\def\qno{\mathfrak{\lcase{q}}_n^{\ol{1}}}

\def\Uqn{\bsU(\mathfrak{\lcase{q}}_n)}
\def\UqnZ{{\Uqn}_{\ZZ}}

\def\Qvs(#1){\mathcal{Q}_{\lcase{v}}{(\lcase{#1})}}
\def\qn{\mathfrak{\lcase{q}}_{n}}
\def\Uvqn{{\boldsymbol U}_{\!{v}}(\mathfrak{\lcase{q}}_{n})}
\def\Uvsln{{\boldsymbol U}_{\!{v}}(\mathfrak{\lcase{sl}}_{n})}

\def\Uz{\bs{U}_{\ZZ}}
\def\Uv{\bs{U}_{v}}

\def\Uvq(#1){U_{\lcase{v}}(\mathfrak{\lcase{q}}_{\lcase{#1}})}

\def\USN(#1){{\US[#1]}_{v}}

\def\SABJR(#1,#2,#3,#4){({#1}|{#2})[\bs{#3}, #4]}
\def\SABJRS(#1,#2,#3,#4){({#1}|{#2})[#3, #4]}
\def\SABJS(#1,#2,#3){({#1}|{#2})[#3]}
\def\SAJRS(#1,#2,#3){{#1}[#2, #3]}
\def\SAJS(#1,#2){{#1}[#2]}
\def\SABJ(#1,#2,#3){({#1}|{#2})[\bs{#3}]}
\def\SAJR(#1,#2,#3){{#1}[\bs{#2}, #3]}
\def\SAJ(#1,#2){{#1}[\bs{#2}]}
\def\ABJR(#1,#2,#3,#4){({#1}|{#2})(\bs{#3}, #4)}
\def\ABJRS(#1,#2,#3,#4){({#1}|{#2})(#3, #4)}
\def\ABJS(#1,#2,#3){({#1}|{#2})(#3)}
\def\AJRS(#1,#2,#3){{#1}(#2, #3)}
\def\AJS(#1,#2){{#1}(#2)}
\def\ABJ(#1,#2,#3){({#1}|{#2})(\bs{#3})}
\def\AJR(#1,#2,#3){{#1}(\bs{#2}, #3)}
\def\AJ(#1,#2){{#1}(\bs{#2})}

\def\STDUE(#1,#2){({#1}+E_{{#2},{#2}+1}-E_{{#2}+1,{#2}+1}|0)}
\def\STDUO(#1,#2){({#1}-E_{{#2}+1,{#2}+1}|E_{{#2},{#2}+1})}
\def\STDLE(#1,#2){({#1}-E_{{#2},{#2}}+E_{{#2}+1,{#2}}|0)}
\def\STDLO(#1,#2){({#1}-E_{{#2},{#2}}|E_{{#2}+1,{#2}})}
\def\STDDE(#1){(D_{#1}|0)}
\def\STDDO(#1,#2){({#1}-E_{{#2},{#2}}|E_{{#2},{#2}})}

\newcommand{\where}{\ \bs{|} \ }

\def\intd(#1,#2,#3){\left[\begin{matrix}{#1};{#2}\\{#3}\end{matrix}\right]}
\def\intds(#1,#2){\left[\begin{matrix}{#1}\\{#2}\end{matrix}\right]}
\def\intdss(#1,#2){\intd({#1},{0},{#2})}

\def\TAIJ(#1,#2){ T^\lhd_{({#1}, {#2})} }
\def\TDIJ(#1,#2){ T^\rhd_{({#1}, {#2})} }
\def\rmText#1{}
\def\rmForm#1{}

\def\AK(#1,#2){ {\overleftarrow{\bf r}}^{#2}_{#1} }
\def\BK(#1,#2){ {\overrightarrow{\bf r}}^{#2}_{#1} }

\newcommand{\refdot}{\bs{\cdot}}

\def\bsU{\boldsymbol U}

\def\sfE{\mathsf E}
\def\sfF{\mathsf F}
\iffalse

\def\genE{\hlt{\sfE}}
\def\genF{\hlt{\sfF}}
\def\genK{\hlt{\mathsf{K}}}
\else

\def\genE{{\sfE}}
\def\genF{{\sfF}}
\def\genXE{{\sfE}}
\def\genXF{{\sfE}}
\def\genK{{\mathsf{K}}}
\def\gene{{\sf{e}}}
\def\genf{{\sf{f}}}
\def\genxe{{\sf{e}}}
\def\genxf{{\sf{e}}}
\def\genxh{{\sf{h}}}
\fi
\newcommand{\vx}{\bs{\mathrm{x}}}

\newif\ifdetail
\detailfalse

\allowdisplaybreaks[1]

\title[Braid Group Action and Quantum Queer Superalgebra ]
{Braid Group Action and Quantum Queer Superalgebra }

\address{Jianmin Chen, School of Mathematical Sciences, Xiamen University, Xiamen 361005, China}
\email{chenjianmin@xmu.edu.cn}
\address{Zhenhua Li,  Jiyang,  Jinan 251400, China}
\email{zhen-hua.li@qq.com}
\address{Hongying Zhu, School of Mathematical Sciences, Xiamen University, Xiamen 361005, China}
\email{hongyingz@stu.xmu.edu.cn}
\keywords{quantum  queer  superalgebra, braid group action, root vector, PBW   basis}
\subjclass[2020]{17B37, 17A70, 20G42, 20C08}

\begin{document}

\maketitle

\begin{center}
  \normalsize
\renewcommand{\thefootnote}{\fnsymbol{footnote}}
 Jianmin Chen, 
 Zhenhua Li and 
Hongying Zhu\footnotemark[1]
\footnotetext[1]{Corresponding author} 
\renewcommand{\thefootnote}{\arabic{footnote}}
    
\end{center}

\begin{abstract}
In this paper, we present explicit actions of braid group on the universal enveloping superalgebra $\Uqn$ and the quantum queer superalgebra $\Uvqn$. 
Then we provide a new  definition of root vectors 
and some explicit expression for them.
With these procedures, we obtain the PBW-type basis containing the product of root vectors.

\end{abstract}

\maketitle

\tableofcontents


\section{Introduction}\label{sec_introduction}
  
  The queer Lie superalgebra $\qn$ is a super analogue of the general linear Lie algebra $\mathfrak{gl}_n$ and it differs drastically from the basic classical Lie superalgebras. 
  One of the major difference lies in that  the Cartan subalgebra of $\qn$ is not purely even and there is no invariant bilinear form on $\qn$, which results in very interesting phenomenon on the highest weight space of a finite dimensional irreducible supermodule (see \cite[\S 1.5.4]{CW}). 
  For this reason, it is very complicated and challenging task to investigate the structure and representation theory of queer Lie superalgebra $\qn$(see \cite{Sergeev1985THETA}), which has received extensive attention and research.

  Quantum groups first arose in the physics literature. In the work of L.D.Faddeev (see \cite{Faddeev}), quantum groups had been developed to construct and solve ` integrable ' quantum systems. 
  In the original form, quantum groups are associative algebras whose defining relations are quantum R-matrix. It was realized independently by V. G. Drinfel'd (see \cite{Uvsln}, \cite{DV}) and M. Jimbo (see \cite{Jimbo}) in the mid-1980s that these algebras are Hopf algebras. In many cases, these algebras are deformations of `universal enveloping algebras' of Lie algebras. 

  Braid groups first appeared, albeit in a disguised form, in an article by Hurwitz published in 1891 and devoted to ramified coverings of surfaces.
  The notion of a braid was explicitly introduced by Artin \cite{artin1925theory} in the 1920s to formalize topological objects that model the intertwining of several strings in Euclidean 3-space. Artin pointed out that braids with a fixed number $n$ of strings form a group, called the $n$-th braid group and denoted by $\fc{B}$. And it had proved by F. Bohnenblust in \cite{Bohnenblust}.  
  Since then, the braids and the braid groups have been extensively studied by topologists and algebraists, which led to a rich theory. Numerous ramifications 
  now play a role in various parts of mathematics, including knot theory and low dimensional topology. And then the braid group action on quantum group was first formulated as a tool by Lusztig in \cite{Lus}. 

  The root vectors of an algebra are typically associated with the root system of a Lie algebra. In the theory  of root system of a Lie algebra, root vectors are a set of specific vectors labeled by roots, which are elements of a Cartan subalgebra.
  Specifically, given a semisimple Lie algebra, which includes a Cartan subalgebra, the root system consists of a set of linearly independent elements taking values in this Cartan subalgebra. Each root corresponds to a root vector, and these root vectors constitute the fundamental representation of the Lie algebra.

  The Poincar\'{e}-Birkhoff-Witt (PBW) theorem was initially formulated and independently proven by Birkhoff \cite{Birkhoff} and Witt \cite{Witt} in 1937, which is often used as a springboard for intevestigating the representation theory of algebras. 
  The classical PBW Theorem associates to any order basis of a Lie algebra is a basis of its universal enveloping algebra. In the case of a quantized universal enveloping algebra, one has no underlying Lie algebra, much less a basis of it, so it is not immediately clear how to obtain an analogue of the PBW theorem for quantum algebra. 
  
  The quantized universal enveloping algebra $\Uvsln$ was introduced by Drinfel'd \cite{Uvsln}. And the braid group action on its algebraic structures, particularly its introduction in the context of quantum groups, was first formulated by Lusztig in 1988 (see \cite{Lus}). 
  And then Lusztig \cite{Lus2} gave analogues of `root vectors' associated with the non-simple roots to construct a basis of $\Uvsln$ of PBW-type.
  In this work, Lusztig developed the concept of the braid group action and root vectors on quantum algebras, which plays a significant role in the study of quantum groups and their representation theory. 

   A quantum deformation $\Uvqn$ of the universal enveloping algebra $\Uqn$ was constructed by Olshanski \cite{Ol} using a modification of the Reshetikhin-Takhtajan-Faddeev method. 
   In this paper, we use the definition of {$\Uvqn$} given in \cite{DW1}, which is equivalent to the one given by Olshanski.
   And then we will present a construction of the braid group action on the quantum queer superalgebra $\Uvqn$, which will enable the definition of root vectors associated to arbitrary root of $\qn$. And the idea for the definition is already described in \cite{CP}: the non-simple roots vectors of $\qn$ can be defined by using the action of a finite covering of the Weyl group on $\qn$; then in the quantum case, the action of the Weyl group is replaced by that of the braid group. 
   Moreover, $\Uqn$ is the classical limit of $\Uvqn$ (see \cite{DJ}), then we have some relevant results for the classical algebraic structures. 
   
   This paper is organized as follows.  
   We  recall the definition of the quantized universal enveloping algebra $\Uvsln$ and the braid group action on it in section \ref{braid}. 
   In section \ref{braid_1},  with the basics of the queer Lie superalgebra $\qn$ and its universal enveloping superalgebra $\Uqn$, we provide an action of braid group in $\Uqn$.  After that we introduce the root vector for $\Uqn$, compute all commutation formulas for these vectors and discuss all the expression of root vectors for $\Uqn$ in section \ref{root_2}.
In section \ref{PBW_0}, we give the PBW-type basis of the Kostant $\ZZ$-form.
   From section \ref{braid_uvqn} onwards, we investigate the quantum case. 
   In section \ref{braid_uvqn}, we recall the equivalent definition of quantum queer superalgebra and present a construction of the braid group action on it.
   In section \ref{root_1}, we introduce quantum root vectors for $\Uvqn$.
   In the last section, we construct a PBW-type basis for the quantum queer superalgebra.

  Throughout this paper, let {$\QQ$} be the field of rational numbers, and  $\Qv$ be the field of fractions in one variable $v$ over  {$\QQ$}.
 
\section{The braid group}\label{braid}
  
We recall the definition of the  quantized universal enveloping algebra of type A and the braid group associated to it.

Let {$\mathrm{A}_{n-1} = (a_{i,j}) $} be the Cartan matrix of type A with order ${n-1}$.
More precisely,
\begin{equation*}
a_{i, j} =
\left\{
\begin{aligned}
&2  \qquad &\mbox{if}\enspace  i=j,   \\
& -1  \qquad &\mbox{if}\enspace  |i-j| = 1,    \\
&0  \qquad & \mbox{if}\enspace  |i-j| > 1. 
\end{aligned}
\right.
\end{equation*}
It is known that {$\f{sl}_{n}$} is the  simple Lie algebra  associated to {$\mathrm{A}_{n-1}$}.
\begin{defn}\label{def_uvsln}
(See \cite{Uvsln} and \cite[Section 9.1]{CP})
The quantum algebra  {$\Uvsln$} is the associative algebra over  {$\Qv$} 
 generated by  $E_{i}$,  $F_{i}$ and {$K_{i}$} ($1 \le i \le n-1$) 
with  relations 
\begin{align*}
&K_iK_j=K_jK_i,\qquad
K_iK_i^{-1}=K_i^{-1}K_i=1;\\
&K_iE_j=v^{a_{i,j}}E_jK_i, \qquad K_iF_j=v^{-a_{i,j}}F_jK_i;\\
&E_iF_i-F_iE_i=\delta_{i,j}\frac{K_i-K_i^{-1}}{v-v^{-1}};\\
&E_iE_j=E_jE_i,\qquad  \qquad F_iF_j=F_jF_i \quad \mbox{ for}\enspace |i-j|> 1;\\
&E_i^{2}E_j-(v+v^{-1})E_iE_jE_i+E_jE_i^{2}=0,\\
&F_i^{2}F_j-(v+v^{-1})F_iF_jF_i+F_jF_i^{2}=0 \quad  \mbox{ for}\enspace |i-j| = 1.
\end{align*}
\end{defn}

The braid group was first introduced by Artin (see \cite{artin1925theory}). In this paper, we use the definition of braid group given in \cite{Lus}.

\begin{defn}
  The braid group {$\fc{B}=\fc{B}(\f{sl}_n)$} associated to  {$\f{sl}_n$} is generated by {$T_1$}, {$\cdots$}, {$T_{n-1}$}  	
with  relations 
\begin{equation}\label{T_relation}
\begin{aligned}
  &T_iT_{j}T_i=T_{j}T_iT_{j}\quad \mbox{if}\enspace  |i-j|= 1,\\
  &T_iT_j=T_jT_i \quad \mbox{if} \enspace|i-j|> 1.
\end{aligned}
\end{equation}  
\end{defn}

Notice that the Weyl group of {$\f{sl}_n$} is identified with the symmetric group {$\fS_{n}$},
and the defining relations of  {$\fc{B}$} are the defining relations of  {$\fS_{n}$} by replacing {$s_i$} with {$T_i$},
while the relation {$s_i^2 = 1$} being removed. 

The following proposition
describes the action of {$\fc{B}$} on {$\Uvsln$} in \cite{Lus}.

\begin{prop}\label{braid_on_gn}
\cite[Proposition 5.1]{Lus}
For  $1 \le i,j \le n-1$, 
the braid group {$\fc{B}$}  acts by {$\Qv$}-algebra automorphisms on {$\Uvsln$} as
\begin{align*}
&
T_{i} (K_{i}) =  K_{i}^{-1},\quad
T_{i} (E_{i}) = -F_{i} K_{i}, \quad
T_{i} (F_{i}) = - K_{i}^{-1} F_{i};\\
&
T_{i} (K_{j}) =  K_{j}  K_{i}  \quad { for } |i - j| = 1 ;\\
& T_{i} (E_{j}) 
=   -  E_{i} E_{j}  +   {v}^{-1} E_{j} E_{i}, \quad
T_{i} (F_{j}) 
= -   F_{j} F_{i} +   {v} F_{i}  F_{j}  \quad \mbox{ for}\enspace  |i - j| = 1; \\
&
T_{i} (K_{j}) =  K_{j}  , \quad
T_{i} (E_{j}) = E_{j}, \quad
T_{i} (F_{j}) = F_{j} 
 \quad \mbox{ for}\enspace |i - j| > 1.
\end{align*}
\end{prop}

Assume  $w \in \fS_n$
is decomposed into a product of simple reflections {${w} = s_{i_1}s_{i_2}\dots s_{i_t} $}, 
it is called  a reduced expression (or a reduced decomposition) if $t$ is the minimal number of simple reflections which appear in any such expression of $w$, 
and $t$ is denoted by $l(w)$, called the length of $w$.

For the reduced expression {${w} = s_{i_1}s_{i_2}\dots s_{i_t} $},
we denote {$T_{w} := T_{i_1}T_{i_2}\dots T_{i_t}$}.
Assume  that $w=s_{i_1}s_{i_2}\dots s_{i_t}$ and $w=s_{j_1}s_{j_2}\dots s_{j_t}$ are two different reduced expressions,
by \cite[Theorem 6.45]{DDPW}, 
we have
\begin{equation}\label{br_relation1}
\begin{aligned}
 T_{i_1}T_{i_2}\dots T_{i_t}=T_{j_1}T_{j_2}\dots T_{j_t}.
\end{aligned}
\end{equation}
 So   {$T_{w}$} is well defined. 
 Referring to \cite[Proposition 8.1.3]{CP}, for any $w_1, w_2 \in \fS_{n}$, 
 we have 
 \begin{equation}\label{br_relation2}
\begin{aligned}
T_{w_1}T_{w_2}=T_{w_1 w_2}, \qquad \mbox{ if }l(w_1)+l(w_2)=l(w_1w_2).
\end{aligned}
\end{equation}

\section{The queer superalgebra $\Uqn$ and the action of $\fc{B}$ on  it}\label{braid_1}

In this section, we will study  the braid group action on queer superalgebra $\Uqn$. 

 Following \cite{DW}, let {$I(n|n) = \{1,  2,  \cdots,  n, -1, -2, \cdots,  -n  \}$}, and the elementary matrices are accordingly denoted by $E_{i,j}$, with $i,j\in I(n|n)$.

The queer  Lie superalgebra  {$\qn$} of order $n$ is the Lie superalgebra 
consists of the matrices of the form 
\begin{equation}\label{qndef}
\begin{aligned}
\left(
\begin{matrix}
A & B\\
B & A 
\end{matrix}
\right),
\end{aligned}
\end{equation}
where {$A, B$} are {$ n \times n$} matrices, 
with  rows and columns  labeled by {$I(n|n)$}.
 The even  part {$\qne$} of {$\qn$} consists of the matrices of the form \eqref{qndef} with {$B=0$},
while the odd part {$\qno$} of {$\qn$} consists of the matrices  with {$A=0$}.

For any {$i, j \in I(n| n)$},  let {$E_{i, j} \in \qn$} be  the matrix  with $1$ at the {$(i, j)$} position  and 
$0$ elsewhere. 
 It is clear that the set {$\{ E_{i, j} \where  i, j   \in I(n| n)\} $} forms a basis of {$\qn$}.
The parity of {$E_{i, j}$} is defined to be 
\begin{align*}
p(E_{ i, j} )=  
\left\{
\begin{matrix}
0 \quad \mbox{ if}\enspace  i j > 0,  \\
1 \quad \mbox{ if}\enspace i j < 0.
\end{matrix}
\right.
\end{align*}

 Let {$\frh = \SE{\frh} \oplus \SO{\frh}$} be the standard Cartan subalgebra of {$\qn$}.
Denote
\begin{align*}
h_{i} = E_{i,i} + E_{-i, -i}, \quad
h_{\ol{i}} = E_{-i,i} + E_{i, -i},  \quad
 i \in \{1,2, \cdots, n\}.
\end{align*}
It is known that {$ \SE{\frh} $} has basis {$\{ h_{1}, \cdots, h_{n}\}$},
while  {$ \SO{\frh} $} has basis {$\{ h_{\ol{1}}, \cdots, h_{\ol{n}}\}$}.

Let {$\{ \bs{\ep}_1, \cdots,  \bs{\ep}_n \} $} be the basis of  {$ {(\SE{\frh})}^* $}   
dual to {$\{ h_{1}, \cdots, h_{n}\}$}.
The dot product  on {$\{ \bs{\ep}_1, \cdots, \bs{\ep}_n \} $} 
 is defined as $\bs{\bs{\ep}}_i  \refdot \bs{\bs{\ep}}_j=\delta_{i,j}$.

For any {$1 \le i \ne j \le n$}, denote 
\begin{equation}
 \alpha_{i,j} :=  \bs{\ep}_i - \bs{\ep}_j .
\end{equation}
According to \cite[Section 1.2.6]{CW},
the standard Borel subalgebra {$\frb$} of {$\qn$} consists of the matrices of the form \eqref{qndef} 
with $A$, $B$ being upper triangular,
and the set of all positive roots corresponding to
 {$\frb$} is
\begin{align*}
\Phi^+ = \{ \alpha_{i,j}  =  \bs{\ep}_i - \bs{\ep}_j  \where 1 \le i < j \le n\} ,
\end{align*}
and the set of all negative roots  is   
\begin{align*}
\Phi^- = \{ \alpha_{i,j}  = \bs{\ep}_i - \bs{\ep}_j  \where 1 \le i < j \le n \} .
\end{align*}
Denote  {$\alpha_j =\alpha_{j, j+1}  =  \bs{\ep}_j -  \bs{\ep}_{j+1}$} for all {$j \in [1, n-1]$},
then the set of all simple positive roots could be denoted as  
\begin{align*}
\Pi^+ = \{ \alpha_{j}    \where 1 \le  j \le n-1\} .
\end{align*}
Denote {$\mathrm{Q} := \bigoplus_{j=1}^{n-1} \ZZ \alpha_{j} $} to be  the set of root lattice.

According to \cite[Section 1.2.6]{CW}, 
the Weyl group of {$\qn$} is identified with the symmetric group {$\fS_n$},
which is also the Weyl group of  {$\f{sl}_{n}$}.
As there is an isomorphism 
\begin{align*}
   \SE{\frh}& \longrightarrow {(\SE{\frh})}^*, \\
   h_i & \mapsto \bs{\ep}_i, \qquad
   h_i - h_{i+1}   \mapsto \alpha_i,  
\end{align*}
we get an action of {$\fS_{n}$} on {$ {(\SE{\frh})}^*$} as
\begin{equation}\label{actionsn} 
\begin{aligned}
s_{i} (\xi) = \xi - \xi(h_i - h_{i+1} ) \alpha_{i} \qquad
( \xi \in {(\SE{\frh})}^*).
\end{aligned}
\end{equation}
It is verified that
\begin{equation}\label{action_dual}
s_{i} (\alpha_j) =
\left\{
\begin{aligned}
& -\alpha_{j} \qquad &\mbox{ for}\enspace i = j,  \\
& \alpha_{i} + \alpha_{j} \qquad &\mbox{ for}\enspace |i-j| = 1,   \\
& \alpha_{j} \qquad & \mbox{ otherwise} ,
\end{aligned}
\right.
\end{equation}
and 
\begin{equation}\label{action_dual_ep}
s_{i} (\bs{\ep}_j) = \bs{\ep}_{s_i(j)}.
\end{equation}

The queer superalgebra {$\Uqn$} is the universal enveloping superalgebra of  $\qn$.
For any homogeneous $x,y \in \Uqn$,  the Lie superbracket is defined as 
\begin{align*}
[x,y] = xy - {(-1)}^{ p(x) p(y) }yx,
\end{align*}
where for homogeneous $z$, $p(z)$ is the parity of $z$.

Following \cite[Proposition 1.1]{DJ},  $\Uqn$   could also be presented as an associative algebra  as follows.
\begin{prop}  \label{classdefine} 	
 The universal enveloping superalgebra $\Uqn$ is the associative superalgebra over $\QQ$ 
generated by even generators $h_i, e_j, f_j$, and odd generators $h_{\ol{i}},e_{\ol{j}},f_{\ol{j}}$, 
with $1\le i \le n$ and $1\le j \le n-1$,
 subjecting to the following relations:
\begin{align*}
   	({\rm QS1})\quad
   	&[h_i,h_j]=0,\qquad [h_i,h_{\ol{j}}]=0,\qquad [h_{\ol{i}},h_{\ol{j}}]=\delta_{i,j}2h_i;\\
   	({\rm QS2})\quad
   	&[h_i,e_j]=(\ep_i,\alpha_{j})e_j, \qquad [h_i,e_{\ol{j}}]=(\ep_i,\alpha_{j})e_{\ol{j}}; \\ 
   	&[h_i,f_j]=-(\ep_i,\alpha_{j})f_j, \qquad [h_i,f_{\ol{j}}]=-(\ep_i,\alpha_{j})f_{\ol{j}}, \\
   	({\rm QS3})\quad
   	&[h_{\ol{i}},e_j]=(\ep_i,\alpha_{j})e_{\ol{j}},\qquad 
   	[h_{\ol{i}},f_j]=-(\ep_i,\alpha_{j})f_{\ol{j}},\\
   	&[h_{\ol{i}},e_{\ol{j}}]=
\begin{cases}
   	e_j \quad &\mbox{if}\enspace i=j \enspace \mbox{or}\enspace j+1, \\
   	0 \quad & otherwise,
\end{cases}
   	\qquad 
   	[h_{\ol{i}},f_{\ol{j}}]=
\begin{cases}
   	f_j \quad &\mbox{if}\enspace i=j \enspace \mbox{or}\enspace j+1, \\
   	0 \quad & otherwise;
\end{cases}\\
   	({\rm QS4})\quad
   	&[e_i,f_j]=\delta_{i,j}(h_{i}- h_{i+1}),\qquad 
   	[e_{\ol{i}},f_{\ol{j}}]=\delta_{i,j}(h_{i}+ h_{i+1}),\\
   	&[e_{\ol{i}},f_j]=\delta_{i,j}(h_{\ol{i}}- h_{\ol{i+1}}),\qquad 
   	[e_i,f_{\ol{j}}]=\delta_{i,j}(h_{\ol{i}}- h_{\ol{i+1}});\\  	
   	({\rm QS5})\quad
   	&[e_i,e_{\ol{j}}]=[e_{\ol{i}},e_{\ol{j}}]=[f_i,f_{\ol{j}}]=[f_{\ol{i}},f_{\ol{j}}]=0 \quad \mbox{for}\enspace |i-j|\ne 1,\\
   	&[e_i,e_j]=[f_i,f_j]=0 \quad \mbox{for} \enspace|i-j|> 1,\\
    &[e_i,e_{i+1}]=[e_{\ol{i}},e_{\ol{i+1}}],\qquad 
   	[e_i,e_{\ol{i+1}}]=[e_{\ol{i}},e_{i+1}],\\
   	&[f_{i+1},f_i]=[f_{\ol{i+1}},f_{\ol{i }}],\qquad 
   	[f_{i+1},f_{\ol{i}}]=[f_{\ol{i+1}},f_{i}];\\
   	({\rm QS6})\quad
   	&[e_i,[e_i,e_j]]=[e_{\ol{i}},[e_i,e_j]]=0,\qquad
   	[f_i,[f_i,f_j]]=[f_{\ol{i}},[f_i,f_j]]=0 \quad \mbox{for}\enspace |i-j|= 1.
\end{align*}
\end{prop}

\begin{rem}
There is an anti-involution $\omega $  of $\Uqn$ defined on the generators as
\begin{align*}
 	&\omega (h_i) = h_i,\quad 
 	\omega (e_j) = f_j,\quad 
 	\omega (f_j) = e_j,\quad \\
 	&\omega (h_{\ol{i}}) = h_{\ol{i}} ,\quad 
 	\omega (e_{\ol{j}} ) = f_{\ol{j}},\quad 
 	\omega (f_{\ol{j}} )= e_{\ol{j}},
\end{align*}
 	with $1\le i \le n$ and $1\le j \le n-1 $.
\end{rem}

Considering Proposition \ref{braid_on_gn},
it is natural to think about the action of $\fcB$ on  $\Uqn$.
\begin{prop}\label{action_2}
There is an action of the braid group $\fcB$ by  superalgebra automorphism of $\Uqn$ 
on the generators defined as 
\begin{align*}
    & T_{i}(h_{j}) = h_{s_{i}(j)}, \qquad
    T_{i}(h_{\ol{j}}) = h_{\ol{s_{i}(j)}} 
    ; \\
    	&
	T_{i} ({e}_{i}) = - {f}_{i} , \quad
	T_{i} ({f}_{i}) = -   {e}_{i}  ,\quad T_{i} ({e}_{\ol{i}}) =- {f}_{\ol{i}}, \quad
	T_{i} ({f}_{\ol{i}}) = -{e}_{\ol{i}};\\
	& 
	T_{i} ({e}_{j}) 
	=   -  {e}_{i} {e}_{j}  +    {e}_{j} {e}_{i}, \quad
	T_{i} ({f}_{j}) 
	=  -   {f}_{j} {f}_{i} +   {f}_{i}  {f}_{j} ,  \\
 & 
	T_{i} ({e}_{\ol{j}}) 
	=   -  {e}_{i} {e}_{\ol{j}}  +    {e}_{\ol{j}} {e}_{i}, \quad
	T_{i} ({f}_{\ol{j}}) 
	= -   {f}_{\ol{j}} {f}_{i} +    {f}_{i}  {f}_{\ol{j}} 
	\quad \mbox{ for}\enspace |i - j| = 1; \\
	&
	T_{i} ({e}_{j}) = {e}_{j} , \quad
	T_{i} ({f}_{j}) = {f}_{j} ,\quad
	T_{i} ({e}_{\ol{j}}) = {e}_{\ol{j}}, \quad
	T_{i} ({f}_{\ol{j}}) = {f}_{\ol{j}} 
	\quad \mbox{ for}\enspace |i - j| > 1.
\end{align*}
  And the action of {$T_{i} ^{-1}$} is given by 
\begin{align*}
&   
       T_{i}^{-1} (h_{j}) = h_{s_{i}(j)}, \qquad
       T_{i}^{-1} (h_{\ol{j}}) = h_{\ol{s_{i}(j)}} ; \\
&
	T_{i} ^{-1}(e_{i}) =  - f_{i}  , \quad
	T_{i} ^{-1}(f_{i}) = -  e_{i} ,\quad 
        T_{i}^{-1}({e}_{\ol{i}}) =- {f}_{\ol{i}}, \quad
	T_{i}^{-1}({f}_{\ol{i}}) = -{e}_{\ol{i}};\\
 & 
	T_{i} ^{-1}(e_{j}) 
	=   -  e_{j} e_{i}   +   e_{i} e_{j}, \quad
	T_{i}^{-1} (f_{j}) 
	= -    f_{i} f_{j}+   f_{j} f_{i} , \\
 & 
	T_{i}^{-1}({e}_{\ol{j}}) 
	=   -  {e}_{j} {e}_{\ol{i}}  +    {e}_{\ol{i}} {e}_{j} , \quad
	T_{i}^{-1} ({f}_{\ol{j}}) 
	= -  {f}_{\ol{i}} {f}_{j} + {f}_{j}  {f}_{\ol{i}}    
	\quad \mbox{ for}\enspace |i - j| = 1; \\
&
	T_{i}^{-1} (e_{j}) = e_{j}, \quad
	T_{i}^{-1} (f_{j}) = f_{j}, \quad
	T_{i}^{-1} ({e}_{\ol{j}}) = {e}_{\ol{j}}, \quad
	T_{i}^{-1} ({f}_{\ol{j}}) = {f}_{\ol{j}} 
	\quad \mbox{ for}\enspace |i - j| > 1.
\end{align*}	
\end{prop}

\begin{proof}
The proof is similar to Theorem \ref{action_uvqn} in Section \ref{braid_uvqn},  so we give a sketch here.
And by a direct calculation, 
it is easy to verify that each $T_{i}$ holds the relations in (QS1)-(QS6).
On the other hand, 
the action of $T_i$'s on  each generator of $\Uqn$ satisfy the relations in \eqref{T_relation}.
%
\end{proof}

\begin{rem}\label{classicalnote2}
For any $1\leq i,j \leq n-1$, $\ |i-j|=1$, direct calculation shows that
\begin{align*}
 	T_iT_j(e_i)=e_j,\quad
 	T_iT_j(e_{\ol{i}})=e_{\ol{j}},\quad 
 	T_iT_j(f_i)=f_j,\quad 
 	T_iT_j(f_{\ol{i}})=f_{\ol{j}}.
\end{align*}	
\end{rem}
 
 Following \cite[Proposition 8.20]{JCJ} and with the same discussion,
we have the following.
\begin{cor}\label{classical_1}
Assume $w \in \fS_n$ and $\alpha_{i}, \alpha_{k} \in \Pi^+ $, such that $w(\alpha_{i})=\alpha_{k}$.
Then we have  
\begin{align*}
	T_w (e_i)=e_k,\quad
	T_w (e_{\ol{i}})=e_{\ol{k}},\quad
	T_w (f_i)=f_k,\quad
	T_w (f_{\ol{i}})=f_{\ol{k}}.
\end{align*}   
\end{cor}


\section{The root vectors of $\Uqn$}\label{root_2}

In this section, we will study the root vectors of the queer superalgebra $\Uqn$.
Let's start with the relations between the positive roots and the longest element $w_0$ in the Weyl group $\fS_n$ of $\qn$.

We denote $N = l(w_0)$.
For convenience, 
for any {$t \ge 0$} 
and  {$\bs{\sigma}  = ({i_1}, {i_2}, \cdots {i_t})  \in \NN^{t}$}, we  denote {$w^{\bs{\sigma}} = s_{i_1}s_{i_2} \cdots s_{i_t}$}.
Notice that $w^{\bs{\sigma}} = 1$ for the case  {$t = 0$}. 
When  {$\bs{\sigma}  = ({i_1}, {i_2}, \cdots {i_N})  \in \NN^{N}$} and 
 $w_0 = w^{\bs{\sigma}} = s_{i_1}s_{i_2} \cdots s_{i_N}$  is a reduced expression,
as {$\f{sl}_n$} and {$\qn$}  have the same  simple positive roots, 
following \cite{CP}, by the action of the Weyl group on {$\f{q}_n$}, 
the set of  positive roots of {$\f{q}_n$} associated with  ${\bs{\sigma}}$ can be described as
\begin{equation}\label{q_nroot}
\begin{aligned}
\Delta^+_{\bs{\sigma}} = \{
\beta_{\bs{\sigma},1} =\alpha_{i_1},\quad 
\beta_{\bs{\sigma},2} =s_{i_1}(\alpha_{i_2}),\quad
\cdots , \quad
\beta_{\bs{\sigma},N} =s_{i_1}s_{i_2}\dots s_{i_{N-1}}(\alpha_{i_N})  \}.
\end{aligned}
\end{equation}

\begin{rem}
According to \cite{DDPW} and \cite[Section 2]{DW1}, 
\begin{equation}\label{root2}
\begin{aligned}
\sharp \Delta^+_{\bs{\sigma}}  = \sharp \Phi^+ = \frac{n(n-1)}{2}, \quad 
\Delta^+_{\bs{\sigma}}  = \Phi^+
&= \{ \alpha_{i,j}={\bs{\ep}}_i -{\bs{\ep}}_j |1\le i <j \le n \},    	
\end{aligned}
\end{equation}
hence the set {$ \Delta^+_{\bs{\sigma}}  $} does not depend on the choosen of {$ {\bs{\sigma}}  $} and we may denote it as {$ \Delta^+$}.
\end{rem}

\begin{defn}\label{Uqnroot}
	Fix a reduced decomposition 
	$w_0 = s_{i_1}s_{i_2} \cdots s_{i_N}$ 
	of the longest element in the Weyl group of $\f{q}_n$. 
	Define the positive root vectors 
	${\genxe}^{\bs{\sigma}}_{k}$,  ${\ol{\gene}}^{\bs{\sigma}}_{k}$ and negative root vectors ${\genxf}^{\bs{\sigma}}_{k}$, ${\ol{\genf}}^{\bs{\sigma}}_{k}$  of {$\Uqn$} as ($1\leq k\leq N$)
\begin{equation}\label{classicalroot}
\begin{aligned}
	&{\genxe}^{\bs{\sigma}}_{k}
	=T_{i_1}T_{i_2}\cdots T_{{i_k}-1}(e_{i_k}),\quad
	{\ol{{\gene}}}^{\bs{\sigma}}_{k}
	=T_{i_1}T_{i_2}\cdots T_{{i_k}-1}(e_{\ol{i_k}}),
	\\
	&{\genf}^{\bs{\sigma}}_{k}=T_{i_1}T_{i_2}\cdots T_{{i_k}-1}(f_{i_k})
	,\quad
	{\ol{{\genf}}}^{\bs{\sigma}}_{k}
	=T_{i_1}T_{i_2}\cdots T_{{i_k}-1}(f_{\ol{i_k}}) .		
\end{aligned}
\end{equation}
\end{defn}

For any ${\beta}_{\bs{\sigma},t }=\alpha_{i,j} \in \Phi^+ \cup \Phi^-$  ({$1\le i <  j \le n$}),  we can denote
\begin{align*}
  &{\genxe}^{\bs{\sigma}}_{i,j} = {\gene}^{\bs{\sigma}}_{t},  \quad
 \ol{{\genxe}}^{\bs{\sigma}}_{i,j} = \ol{\gene}^{\bs{\sigma}}_{t}, 
\\
& {\gene}^{\bs{\sigma}}_{j,i} = {\genf}^{\bs{\sigma}}_{t},\quad
\ol{{\gene}}^{\bs{\sigma}}_{j,i} =  	\ol{\genf}^{\bs{\sigma}}_{t}.
\end{align*}
 And we denote the set of  positive and negative root vectors as
\begin{equation}\label{eq_basis}	
\begin{aligned}
&\psi^+_{\bs{\sigma}}
 = \{ {\gene}^{\bs{\sigma}}_{t},    {\ol{\gene}}^{\bs{\sigma}}_{t} \where 1 \le t \le N\}
 = \{ {\genxe}^{\bs{\sigma}}_{i,j},   {\ol{\genxe}}^{\bs{\sigma}}_{i,j} \where 1 \le i<j \le N\},\\
&\psi^-_{\bs{\sigma}} 
=  \{  {\genf}^{\bs{\sigma}}_{t}, {\ol{\genf}}^{\bs{\sigma}}_{t}  \where 1 \le t \le N \}
 = \{ {\genxf}^{\bs{\sigma}}_{j,i},   {\ol{\genxf}}^{\bs{\sigma}}_{j,i} \where 1 \le i<j \le N\}.
\end{aligned}
\end{equation}

\begin{exam}
Fix $n=3$, there are two choices for reduced decomposition of $w_0$.
\begin{enumerate}
\item
 If we choose {$\bs{\sigma}$} = (1,2,1) and  $w_0 = w^{\bs{\sigma}} = s_1 s_2 s_1$,
the positive roots are 
\begin{align*}
\beta_{\bs{\sigma}, 1} = \alpha_1 = \alpha_{1,2}, \quad
\beta_{\bs{\sigma}, 2} = s_1(\alpha_2) = \alpha_{1,3}, \quad
\beta_{\bs{\sigma}, 3} = s_1 s_2 (\alpha_1) = \alpha_{2,3},
\end{align*} 
while the corresponding  positive root vectors are
\begin{align*}
&{\genxe}^{\bs{\sigma}}_{1,2}  = e_1 ,
\quad {\genxe}^{\bs{\sigma}}_{1,3}  =  T_1 (e_2) = -e_1 e_2 + e_2 e_1 , 
\quad {\genxe}^{\bs{\sigma}}_{2,3}  = T_1 T_2 (e_1) = e_2 ,\\
&\ol{\genxe}^{\bs{\sigma}}_{1,2}  = e_{\ol{1}} ,
\quad \ol{\genxe}^{\bs{\sigma}}_{1,3} = T_1 (e_{\ol{2}}) = -e_1 e_{\ol{2}} + e_{\ol{2}} e_1 , 
\quad \ol{\genxe}^{\bs{\sigma}}_{2,3} = T_1 T_2 (e_{\ol{1}}) = e_{\ol{2}}.
\end{align*}
\item
If we choose  {$\bs{\gamma}$} = (2,1,2) and  $w_0 = w^{\bs{\gamma}} = s_2 s_1 s_2$, 
then
\begin{align*}
\beta_{\bs{\gamma}, 1}  = \alpha_2 = \alpha_{2,3}, \quad
\beta_{\bs{\gamma}, 2} = s_2(\alpha_1) = \alpha_{1,3}, \quad
\beta_{\bs{\gamma}, 3} = s_2 s_1 (\alpha_2) = \alpha_{1,2},
\end{align*} 
and the corresponding positive root vectors are
\begin{align*}
&{\genxe}^{\bs{\gamma}}_{2,3}  = e_2 ,
\quad {\genxe}^{\bs{\gamma}}_{1,3}  = T_2 (e_1) = -e_2 e_1 + e_1 e_2 , 
\quad {\genxe}^{\bs{\gamma}}_{1,2}  = T_2 T_1 (e_2) = e_1 ,\\
&{\ol{\genxe}}^{\bs{\gamma}}_{2,3}  =   e_{\ol{2}} ,
\quad {\ol{\genxe}}^{\bs{\gamma}}_{1,3}  =   T_2 (e_{\ol{1}}) = -e_2 e_{\ol{1}} + e_{\ol{1}} e_2 , 
\quad {\ol{\genxe}}^{\bs{\gamma}}_{1,2}  =   T_2 T_1 (e_{\ol{2}}) =  e_{\ol{1}}.
\end{align*}
\end{enumerate}

\end{exam}

In the following of this paper, 
we fix $ \bs{\sigma} $ for a   reduced decomposition  of the longest element {$\fS_{n}$} as follows:
\begin{equation}\label{w_0_fix}
\begin{aligned}
& \bs{\sigma} = ( {1}, \quad  {2},  1,  \quad 3, 2, 1,  \quad \cdots ,  \quad n-1,  n-2,  \cdots,  2,  1) \in \NN^N , \\
& w_0 = w^ {\bs{\sigma}}  =(s_1)\refdot(s_2s_1)\refdot(s_3s_2s_1)\refdot\cdots\refdot (s_{n-1}\cdots s_1).
\end{aligned}
\end{equation} 

By a direct computation,
we conclude the following lemma.
\begin{lem}\label{lemroot}
Suppose that $1\le i <  j \le n$. If $j=i+1$, then 
\begin{align*}
\beta_{\bs{\sigma}, \frac{i(i+1)}{2}}=(s_1)\refdot (s_2 s_1)\refdot\cdots \refdot(s_i s_{i-1}\cdots s_2)(\alpha_1)=\alpha_{i,i+1};
\end{align*}
if $j>i+1$, then 
\begin{align*}
\beta_{\bs{\sigma}, \frac{(j-2)(j-1)}{2}+i-1}=(s_1)\refdot (s_2 s_1)\refdot\cdots \refdot (s_{j-1} s_{j-2}\cdots s_{j-i+1})(\alpha_{j-i})=\alpha_{i,j}.
\end{align*}
\end{lem}
    
\begin{proof}
By the group action in (\ref{action_dual_ep}), we see that  
\begin{align*}
&(s_1)\refdot (s_2 s_1)\refdot\cdots \refdot (s_i s_{i-1}\cdots s_2)(\alpha_1)\\
&=(s_1)\refdot (s_2 s_1)\refdot\cdots \refdot(s_i s_{i-1}\cdots s_2)({\bs{\ep}}_1-{\bs{\ep}}_2)\\
&=(s_1)\refdot (s_2 s_1)\refdot\cdots \refdot(s_i s_{i-1}\cdots s_2)({\bs{\ep}}_1)
-(s_1)\refdot (s_2 s_1)\refdot\cdots \refdot (s_i s_{i-1}\cdots s_2)({\bs{\ep}}_2)\\
&={\bs{\ep}}_{(s_1)\refdot (s_2 s_1)\refdot\cdots \refdot (s_i s_{i-1}\cdots s_2)(1)}
-{\bs{\ep}}_{(s_1)\refdot (s_2 s_1)\refdot\cdots \refdot(s_i s_{i-1}\cdots s_2)(2)}\\
&={\bs{\ep}}_{(s_1)\refdot (s_2 s_1)\refdot\cdots\refdot (s_{i-1}\cdots s_1)(1)}
-{\bs{\ep}}_{(s_1)\refdot (s_2 s_1)\refdot\cdots\refdot ( s_{i-1}\cdots s_1)(i+1)}\\
&={\bs{\ep}}_{(s_1)\refdot (s_2 s_1)\refdot\cdots\refdot (s_{i-2}\cdots s_1)(i)}
-{\bs{\ep}}_{i+1}\\	
&={\bs{\ep}}_{i}-{\bs{\ep}}_{i+1}\\
&=\alpha_{i,i+1}.
\end{align*}	
Hence, we have proved the first equation. For the second equation, we can prove it directly. 
For any $1\le i , j \le n$  with $j-i>1$, 
\begin{align*}
&(s_1)\refdot (s_2 s_1)\refdot\cdots \refdot (s_{j-1} s_{j-2}\cdots s_{j-i+1})(\alpha_{j-i})\\
&=(s_1)\refdot (s_2 s_1)\refdot\cdots \refdot (s_{j-1} s_{j-2}\cdots s_{j-i+1})({\bs{\ep}}_{j-i}-{\bs{\ep}}_{j-i+1})\\
&=(s_1)\refdot (s_2 s_1)\refdot\cdots \refdot (s_{j-1} s_{j-2}\cdots s_{j-i+1})({\bs{\ep}}_{j-i})\\
&\qquad-(s_1)\refdot (s_2 s_1)\refdot\cdots \refdot (s_{j-1} s_{j-2}\cdots s_{j-i+1})({\bs{\ep}}_{j-i+1})\\
&={\bs{\ep}}_{(s_1)\refdot (s_2 s_1)\refdot\cdots \refdot (s_{j-1} s_{j-2}\cdots s_{j-i+1})(j-i)}
-{\bs{\ep}}_{(s_1)\refdot (s_2 s_1)\refdot\cdots \refdot (s_{j-1} s_{j-2}\cdots s_{j-i+1})(j-i+1)}\\
&={\bs{\ep}}_{(s_1)\refdot (s_2 s_1)\refdot\cdots \refdot ( s_{j-2}\cdots s_{j-i-1})(j-i)}
-{\bs{\ep}}_{j}\\
&={\bs{\ep}}_{(s_1)\refdot (s_2 s_1)\refdot\cdots \refdot ( s_{j-3}\cdots s_{j-i-2})(j-i-1)}
-{\bs{\ep}}_{j}\\
&\quad \vdots \\
&={\bs{\ep}}_{(s_1)\refdot (s_2 s_1)\refdot\cdots \refdot ( s_{i}\cdots s_{1})(2)}
-{\bs{\ep}}_{j}\\
&={\bs{\ep}}_{(s_1)\refdot (s_2 s_1)\refdot\cdots \refdot ( s_{i-1}\cdots s_{1})(1)}
-{\bs{\ep}}_{j}\\
&={\bs{\ep}}_{i}-{\bs{\ep}}_{j}\\
&=\alpha_{i,j}.
\end{align*}		
\end{proof}

\begin{cor}\label{classical_2_0}
Assume that $1\leq i< n$, $i+1<  j< n$. Then we have
\begin{align*}
&{\genxe}^{\bs{\sigma}}_{i,i+1}=e_i,\quad 
      {\ol{\genxe}^{\bs{\sigma}}_{i,i+1}}=e_{\overline{i}},\quad
      {\genxf}^{\bs{\sigma}}_{i+1,i}=f_i,\quad 
      {\ol{\genxf}^{\bs{\sigma}}_{i+1,i }}=f_{\overline{i}}, \\
&{\genxe}^{\bs{\sigma}}_{i,j}
=-{\genxe}^{\bs{\sigma}}_{i,j-1} {\genxe}^{\bs{\sigma}}_{j-1,j}+ {\genxe}^{\bs{\sigma}}_{j-1,j}{\genxe}^{\bs{\sigma}}_{i,j-1}
     	=[ 
     	   \cdots
     	   [ 
     	     [e_i, -e_{i+1}], -e_{i+2} 
     	              ], 
     	               \cdots ,-e_{j-1} ],\\
&{\ol{\genxe}}^{\bs{\sigma}}_{i,j}
=-{\genxe}^{\bs{\sigma}}_{i,j-1} {\ol{\genxe}}^{\bs{\sigma}}_{j-1,j} + {\ol{\genxe}}^{\bs{\sigma}}_{j-1,j} {\genxe}^{\bs{\sigma}}_{i,j-1}
     	 =[ 
     	 \cdots
     	    [ 
     	          [e_i, -e_{i+1}], -e_{i+2} 
     	                    ], 
     	 \cdots ,-e_{\ol{j-1}} ],\\
&{\genxf}^{\bs{\sigma}}_{j,i}
=-{\genxf}^{\bs{\sigma}}_{j,j-1}  {\genxf}^{\bs{\sigma}}_{j-1,i}+{\genxf}^{\bs{\sigma}}_{j-1,i} {\genxf}^{\bs{\sigma}}_{j,j-1}
     	 =[-f_{j-1}, 
     	       [-f_{j-2}, 
     	          \cdots 
     	           [-f_{i+1}, f_i] 
     	                 ] 
     	                   \cdots
     	                                      ],\\
 &{\ol{\genxf}}^{\bs{\sigma}}_{j,i}
 =-{\ol {\genxf}^{\bs{\sigma}}_{j,j-1}}  {\genxf}^{\bs{\sigma}}_{j-1,i}+{\genxf}^{\bs{\sigma}}_{j-1,i}{\ol {\genxf}^{\bs{\sigma}}_{j,j-1}}
=[-f_{\ol{j-1}}, 
     	      [-f_{j-2}, 
     	        \cdots 
     	          [-f_{i+1}, f_i] 
     	                    ]
     	                       \cdots
     	                             ].
\end{align*}    
\end{cor}

\begin{proof}
By Definition \ref{Uqnroot} and Remark \ref{note_i_i1}, we see that  
\begin{align*}
{\genxe}^{\bs{\sigma}}_{i,i+1}
&=(T_1)\refdot (T_2T_1)\refdot\cdots \refdot (T_{i-1}T_{i-2}\cdots T_1)\refdot (T_iT_{i-1}\cdots T_3T_2)({e}_1)\\
&=(T_1)\refdot (T_2T_1)\refdot\cdots \refdot (T_{i-1}T_{i-2}\cdots T_2)\refdot (T_iT_{i-1}\cdots T_3)\refdot T_1T_2({e}_1)\\
&=(T_1)\refdot (T_2T_1)\refdot\cdots \refdot (T_{i-1}T_{i-2}\cdots T_2)\refdot (T_iT_{i-1}\dots T_3)({e}_2)\\
&=(T_1)\refdot (T_2T_1)\refdot\cdots \refdot (T_{i-1}T_{i-2}\cdots T_3)\refdot (T_iT_{i-1}\cdots T_4)\refdot T_2T_3({e}_2)\\
&=(T_1)\refdot(T_2T_1)\refdot\cdots \refdot (T_{i-1}T_{i-2}\cdots T_3)\refdot(T_iT_{i-1}\cdots T_4)({e}_3)\\
&\quad \vdots \\
&=(T_1)\refdot(T_2T_1)\refdot\cdots \refdot (T_{i-2}\cdots T_1)({e}_i)\\
&={e}_i.
\end{align*}
Then the first equation is proved.
More generally, for any $1\le i , j \le n$ ,$j>i+1$, we have 
\begin{align*}
{\genxe}^{\bs{\sigma}}_{i,j}
&=(T_1)\refdot(T_2T_1)\refdot\cdots \refdot (T_{i-1}\cdots T_1)\refdot (T_i\cdots T_1)\refdot\cdots \refdot (T_{j-1}\cdots T_{j-i+1})({e}_{j-i})\\
&=(T_1)\refdot(T_2T_1)\refdot\cdots \refdot (T_{i-1}\cdots T_2)\refdot [(T_i\cdots T_3)\refdot (T_1 T_2 T_1)]\refdot\cdots \refdot(T_{j-1}\cdots T_{j-i+1})({e}_{j-i})\\
&\quad \vdots \\
&=(T_1)\refdot(T_2T_1)\refdot \cdots \refdot (T_{i-1}\cdots T_2)\refdot(T_i\cdots T_1) \refdot \cdots \refdot (T_{j-1}\cdots T_{j-i+2} )\refdot T_{j-i}T_{j-i+1}({e}_{j-i})\\ 
&=(T_1)\refdot(T_2T_1)\refdot \cdots \refdot (T_{i-1}\cdots T_2)\refdot(T_i\cdots T_1) \refdot \cdots \refdot (T_{j-1}\cdots T_{j-i+2})({e}_{j-i+1})\\ 
&=(T_1)\refdot(T_2T_1)\refdot \cdots \refdot (T_{i-1}\cdots T_3)(T_i\cdots T_1) \refdot \cdots \refdot (T_{j-1}\cdots T_{j-i+3})({e}_{j-i+2})\\  
&\quad \vdots \\
&=(T_1)\refdot(T_2T_1)\refdot \cdots \refdot (T_{i-2}\cdots T_1)\refdot(T_i\cdots T_1) \refdot \cdots \refdot (T_{j-2}\cdots T_{1})({e}_{j-1})\\  
&=T_iT_{i+1}\cdots T_{j-2}({e}_{j-1}). 
\end{align*}
Applying the last equation by replacing $j$ with $j+1$, we have
\begin{align*}
{\genxe}^{\bs{\sigma}}_{i,j+1}
&=(T_iT_{i+1}\cdots T_{j-2})\refdot T_{j-1}({e}_{j})\\
&=(T_iT_{i+1}\cdots T_{j-2})(-{e}_{j-1}{e}_{j}
+ {e}_{j}{e}_{j-1})\\
&=-(T_iT_{i+1}\cdots T_{j-2})({e}_{j-1})\refdot (T_iT_{i+1}\cdots T_{j-2})({e}_{j})\\
 &\quad + (T_iT_{i+1}\cdots
  T_{j-2})({e}_{j})\refdot
  (T_iT_{i+1}\cdots T_{j-2})({e}_{j-1})\\
&=- {\genxe}^{\bs{\sigma}}_{i,j}{e}_{j}
 + {e}_{j}{\genxe}^{\bs{\sigma}}_{i,j}.
\end{align*} 
 And thus we have proved the equation for ${\genxe}^{\bs{\sigma}}_{i,j}$. 
 With a parallel argument, the remaining  formulas could be proved.	
\end{proof}

As a natural consesquence, we will provide super commutation formulas for these root vectors.
\begin{prop}
The following holds for $1\le i,j,k,l\le n$ satisfying $i<j, k<l$ in $\Uqn$:
\begin{enumerate}
\item 
\begin{align*}
&[{\genxe}_{i,j}^{\bs{\sigma}}, {\genxe}_{k,l}^{\bs{\sigma}}]=\left\{\begin{aligned}
&{\genxe}_{k,j}^{\bs{\sigma}}  \quad &\mbox{if }  i=l\enspace\& \enspace j\ne k,   \\
&- {\genxe}_{i,l}^{\bs{\sigma}}  \quad &\mbox{if }  j=k\enspace \& \enspace i\ne l,    \\
&0  \quad & \mbox{otherwise }; 
\end{aligned}\right.
\\
&[{\genxe}_{i,j}^{\bs{\sigma}}, {\genxe}_{l,k}^{\bs{\sigma}}]=\left\{\begin{aligned}
&{\genxe}_{l,j}^{\bs{\sigma}}  \quad &\mbox{if }  i=k\enspace\& \enspace j\ne l,   \\
&- {\genxe}_{i,k}^{\bs{\sigma}}  \quad &\mbox{if }  j=l\enspace \& \enspace i\ne k,    \\
&h_{{i}}-h_{{j}}\quad &\mbox{if }  i=k\enspace \& \enspace j\ne l,    \\
&0  \quad & \mbox{otherwise }; 
\end{aligned}\right.
\\
&[{\genxe}_{j,i}^{\bs{\sigma}}, {\genxe}_{l,k}^{\bs{\sigma}}]=\left\{\begin{aligned}
&-{\genxe}_{j,k}^{\bs{\sigma}}  \quad &\mbox{if }  i=l\enspace\& \enspace j\ne k,   \\
& {\genxe}_{i,l}^{\bs{\sigma}}  \quad &\mbox{if }  j=k\enspace \& \enspace i\ne l,    \\
&0  \quad & \mbox{otherwise }. 
\end{aligned}\right.
\end{align*}

\item
\begin{align*}
&[{\genxe}_{i,j}^{\bs{\sigma}}, {\ol{\genxe}}_{k,l}^{\bs{\sigma}}]=\left\{ \begin{aligned}
&{\ol{\genxe}}_{k,j}^{\bs{\sigma}}  \quad &\mbox{if }  i=l\enspace\& \enspace j\ne k,   \\
&- {\ol{\genxe}}_{i,l}^{\bs{\sigma}}  \quad &\mbox{if }  j=k\enspace \& \enspace i\ne l,   \\
&0  \quad & \mbox{otherwise }; 
\end{aligned}\right.
\\
&[{\genxe}_{i,j}^{\bs{\sigma}}, {\ol{\genxe}}_{l,k}^{\bs{\sigma}}]=\left\{ \begin{aligned}
&{\ol{\genxe}}_{l,j}^{\bs{\sigma}}  \quad &\mbox{if }  i=k\enspace\& \enspace j\ne l,   \\
&- {\ol{\genxe}}_{i,k}^{\bs{\sigma}}  \quad &\mbox{if }  j=l\enspace \& \enspace i\ne k,    \\
&h_{\ol{i}}-h_{\ol{j}}\quad &\mbox{if }  i=k\enspace \& \enspace j\ne l,    \\
&0  \quad & \mbox{otherwise }; 
\end{aligned}\right.
\\
&[{\ol{\genxe}}_{i,j}^{\bs{\sigma}}, {{\genxe}}_{l,k}^{\bs{\sigma}}]=\left\{ \begin{aligned}
&{\ol{\genxe}}_{l,j}^{\bs{\sigma}}  \quad &\mbox{if }  i=k\enspace\& \enspace j\ne l,   \\
&-{\ol{\genxe}}_{i,k}^{\bs{\sigma}}  \quad &\mbox{if }  j=l\enspace \& \enspace i\ne k,    \\
&h_{\ol{i}}-h_{\ol{j}}\quad &\mbox{if }  i=k\enspace \& \enspace j\ne l,    \\
&0  \quad & \mbox{otherwise }; 
\end{aligned}\right.
\\
&[{\genxe}_{j,i}^{\bs{\sigma}}, {\ol{\genxe}}_{l,k}^{\bs{\sigma}}]=\left\{ \begin{aligned}
&-{\ol{\genxe}}_{j,k}^{\bs{\sigma}}  \quad &\mbox{if }  i=l\enspace\& \enspace j\ne k,   \\
& {\ol{\genxe}}_{i,l}^{\bs{\sigma}}  \quad &\mbox{if }  j=k\enspace \& \enspace i\ne l,    \\
&0  \quad & \mbox{otherwise }. 
\end{aligned}\right.
\end{align*}

\item
\begin{align*}
&[{\ol{\genxe}}_{i,j}^{\bs{\sigma}}, {\ol{\genxe}}_{k,l}^{\bs{\sigma}}]=\left\{ \begin{aligned}
&-{{\genxe}}_{k,j}^{\bs{\sigma}}  \quad &\mbox{if }  i=l\enspace\& \enspace j\ne k,   \\
&- {{\genxe}}_{i,l}^{\bs{\sigma}}  \quad &\mbox{if }  j=k\enspace \& \enspace i\ne l,    \\
&0  \quad & \mbox{otherwise }; 
\end{aligned}\right.
\\
&[{\ol{\genxe}}_{i,j}^{\bs{\sigma}}, {\ol{\genxe}}_{l,k}^{\bs{\sigma}}]=\left\{ \begin{aligned}
&-{{\genxe}}_{l,j}^{\bs{\sigma}}  \quad &\mbox{if }  i=k\enspace\& \enspace j\ne l,   \\
&-{{\genxe}}_{i,k}^{\bs{\sigma}}  \quad &\mbox{if }  j=l\enspace \& \enspace i\ne k,    \\
&h_{{i}}+h_{{j}}\quad &\mbox{if }  i=k\enspace \& \enspace j\ne l,    \\
&0  \quad & \mbox{otherwise }; 
\end{aligned}\right.
\\
&[{\ol{\genxe}}_{j,i}^{\bs{\sigma}}, {\ol{\genxe}}_{l,k}^{\bs{\sigma}}]=\left\{ \begin{aligned}
&-{{\genxe}}_{j,k}^{\bs{\sigma}}  \quad &\mbox{if }  i=l\enspace\& \enspace j\ne k,   \\
& -{{\genxe}}_{i,l}^{\bs{\sigma}}  \quad &\mbox{if }  j=k\enspace \& \enspace i\ne l,    \\
&0  \quad & \mbox{otherwise }. 
\end{aligned}\right.
\end{align*}

\item 
\begin{align*}
&[h_k, {{\genxe}}_{i,j}^{\bs{\sigma}}]=\left\{ \begin{aligned}
&{{\genxe}}_{i,j}^{\bs{\sigma}}  \quad &\mbox{if }  k=i,   \\
& -{{\genxe}}_{i,j}^{\bs{\sigma}}  \quad &\mbox{if }  k=j,    \\
&0  \quad & \mbox{otherwise }; 
\end{aligned}\right.
\\
&[h_k, {{\genxe}}_{j,i}^{\bs{\sigma}}]=\left\{ \begin{aligned}
&-{{\genxe}}_{j,i}^{\bs{\sigma}}  \quad &\mbox{if }  k=i,   \\
&{{\genxe}}_{j,i}^{\bs{\sigma}}  \quad &\mbox{if }  k=j,    \\
&0  \quad & \mbox{otherwise }; 
\end{aligned}\right.
\\
&[h_k, {\ol{\genxe}}_{i,j}^{\bs{\sigma}}]=\left\{ \begin{aligned}
&{\ol{\genxe}}_{i,j}^{\bs{\sigma}}  \quad &\mbox{if }  k=i ,   \\
& -{\ol{\genxe}}_{i,j}^{\bs{\sigma}}  \quad &\mbox{if }  k=j,    \\
&0  \quad & \mbox{otherwise }; 
\end{aligned}\right.
\\
&[h_k, {\ol{\genxe}}_{j,i}^{\bs{\sigma}}]=\left\{ \begin{aligned}
&-{\ol{\genxe}}_{j,i}^{\bs{\sigma}}  \quad &\mbox{if }  k=i ,   \\
&{\ol{\genxe}}_{j,i}^{\bs{\sigma}}  \quad &\mbox{if }  k=j,    \\
&0  \quad & \mbox{otherwise }. 
\end{aligned}\right.
\end{align*}

\item 
\begin{align*}
&[h_{\ol{k}}, {{\genxe}}_{i,j}^{\bs{\sigma}}]=\left\{ \begin{aligned}
&{\ol{\genxe}}_{i,j}^{\bs{\sigma}}  \quad &\mbox{if }  k=i,   \\
& -{\ol{\genxe}}_{i,j}^{\bs{\sigma}}  \quad &\mbox{if }  k=j,    \\
&0  \quad & \mbox{otherwise };
\end{aligned}\right.
\\
&[h_{\ol{k}}, {{\genxe}}_{j,i}^{\bs{\sigma}}]=\left\{ \begin{aligned}
&-{\ol{\genxe}}_{j,i}^{\bs{\sigma}}  \quad &\mbox{if }  k=i,   \\
& {\ol{\genxe}}_{j,i}^{\bs{\sigma}}  \quad &\mbox{if }  k=j,    \\
&0  \quad & \mbox{otherwise }; 
\end{aligned}\right.
\\
&[h_{\ol{k}}, {\ol{\genxe}}_{i,j}^{\bs{\sigma}}]=\left\{ \begin{aligned}
&{{\genxe}}_{i,j}^{\bs{\sigma}}  \quad &\mbox{if }  k=i \enspace\mbox{or } k=j,   \\
&0  \quad & \mbox{otherwise }; 
\end{aligned}\right.
\\
&[h_{\ol{k}}, {\ol{\genxe}}_{j,i}^{\bs{\sigma}}]=\left\{ \begin{aligned}
&{{\genxe}}_{j,i}^{\bs{\sigma}}  \quad &\mbox{if }  k=i \enspace\mbox{or } k=j,   \\
&0  \quad & \mbox{otherwise }. 
\end{aligned}\right.
\end{align*}
\end{enumerate}
\end{prop}

\begin{proof}
A direct calculation shows that 
\begin{align*}
[{{\genxe}}_{i,j}^{\bs{\sigma}}, f_i]
&=[[ 
     	   \cdots
     	   [ 
     	     [e_i, -e_{i+1}], -e_{i+2} 
     	              ], 
     	               \cdots ,-e_{j-1} ], f_i]\\
&=[ 
     	   \cdots
     	   [[ 
     	     [e_i, -e_{i+1}], f_i], -e_{i+2} 
     	              ], 
     	               \cdots ,-e_{j-1} ]\\
&=[ f_i, [ 
     	   \cdots
     	   [ 
     	     [e_i, -e_{i+1}], -e_{i+2} 
     	              ], 
     	               \cdots ,-e_{j-1} ]]
    -{{\genxe}}_{i+1,j}^{\bs{\sigma}}
\end{align*}
 Where the third equality is due to the relation $[e_i, f_i]=h_i-h_{i+1}$ in Proposition \ref{classdefine}. Similarly, the following holds: 
\begin{align*}
&[{{\genxe}}_{i,j}^{\bs{\sigma}}, f_k]=\left\{ \begin{aligned}
&{{\genxe}}_{i+1,j}^{\bs{\sigma}}  \quad &\mbox{if }  k=i,   \\
& -{{\genxe}}_{i,j-1}^{\bs{\sigma}}  \quad &\mbox{if }  k=j-1,    \\
&0  \quad & \mbox{otherwise },
\end{aligned}\right.
\end{align*}
and hence
 $[{{\genxe}}_{i,j}^{\bs{\sigma}}, {{\genxe}}_{l,i}^{\bs{\sigma}}]={{\genxe}}_{l,j}^{\bs{\sigma}}$ for $j\ne l$ and $[{{\genxe}}_{i,j}^{\bs{\sigma}}, {{\genxe}}_{j,k}^{\bs{\sigma}}]={{\genxe}}_{i,k}$ for $i\ne k$. 
Thus, by Corollary \ref{classical_2_0}, one can deduce that 
\begin{align*}
[{{\genxe}}_{i,j}^{\bs{\sigma}},{{\genxe}}_{j,i}^{\bs{\sigma}}]
&=[{{\genxe}}_{i,j}^{\bs{\sigma}}, -{{\genxe}}_{j,j-1}^{\bs{\sigma}}{{\genxe}}_{j-1,i}^{\bs{\sigma}}+{{\genxe}}^{\bs{\sigma}}_{j-1,i}{{\genxe}}^{\bs{\sigma}}_{j,j-1}]\\
&=[{{\genxe}}_{i,j-1}^{\bs{\sigma}},{{\genxe}}_{j-1,i}^{\bs{\sigma}}]+ h_{j-1}-h_j\\
&=[{{\genxe}}_{i,i+1}^{\bs{\sigma}},{{\genxe}}_{i+1,i}^{\bs{\sigma}}]+ h_{i+1}-h_j\\
&=h_i - h_j.
\end{align*}

Hence, the second formula of part (1) holds. Similarly, the other cases can be verified and we skip the detail.
\end{proof}

It is seen that the root vectors  depend  on the choice of a reduced expression of $w_0$. 
We  then discuss the root vectors of $\Uqn$ for arbitrary reduced expression of $w_0$.
Assume  {$\bs{\lambda} \in \NN^N$} and $w_0 = w^{\bs{\lambda}}$ is a reduced expression of {$w_0$} different  from \eqref{w_0_fix}.
It is known that $w^{\bs{\lambda}}$ could be obtained by applying the following relations 
\begin{equation}\label{sys}
\begin{aligned}
       & s_i s_{i+1} s_i = s_{i+1} s_i s_{i+1}, \\
       & s_i s_j =s_j s_i \quad \mbox{ for} \enspace |i-j|>1.         
\end{aligned}
\end{equation}  
on $w^{\bs{\sigma}}$ in finitely many times.
Assume $w^{\bs{\varsigma}}$ and $w^{\bs{\tau}}$ are two reduced expressions of $w_0$ and 
  $w^{\bs{\tau}}$  could be obtained by applying one of the equations in \eqref{sys}, 
we will  find the relations between $\psi^+_{\bs{\varsigma}} $ and $\psi^+_{\bs{\tau}}$.

\begin{enumerate}
\item[Case 1:]
Assume that $w^{\bs{\varsigma}} =w{'}s_{{i}} s_{{j}}  w{''}$, $w^{\bs{\tau}} =w{'}s_{{j}} s_{{i}}  w{''}$, 
   with $l(w{'}) = h$ and {$ | {i}  -  {j} | > 1 $.} \\

By the action in \eqref{action_dual}, we see that {$s_{{j}}(\alpha_{{i}}) = \alpha_{{i}} $} and
{$s_{{i}}(\alpha_{{j}})=\alpha_{{j}}$}. 
Therefore, the relationships between the roots in the classical case corresponding to the reduced expressions $w^{\bs{\varsigma}}$ and $w^{\bs{\tau}}$ are as follows: 
\begin{equation}\label{classical1_0}
\begin{aligned}
&\beta_{t}^{\bs{\varsigma}}=\beta_{t}^{\bs{\tau}} , \enspace  \mbox{if }\enspace t \ne h+1, h+2 ; \\
&\beta_{h+1}^{\bs{\varsigma}}
=w{'}(\alpha_{{i}})
=w{'} s_{{j}}(\alpha_{{i}})
=\beta_{h+2}^{\bs{\tau}},  \\
&\beta_{h+2}^{\bs{\varsigma}}
=w{'}s_{{i}}(\alpha_{{j}})
=w{'}(\alpha_{{j}})
=\beta_{h+1}^{\bs{\tau}}.
\end{aligned}
\end{equation}
On the other hand, \eqref{br_relation1} implies
\begin{equation}\label{classical1_1}
\begin{aligned}
   &{\gene}_{{t}}^{\bs{\varsigma}}
   ={\gene}_{{t}}^{\bs{\tau}}, \quad 
   \ol{\gene}_{{t}}^{{\bs{\varsigma}}}
   =\ol{\gene}_{{t}}^{\bs{\tau}}, \quad
   {\genf}_{{t}}^{\bs{\varsigma}}
   ={\genf}_{{t}}^{\bs{\tau}}, \quad 
   \ol{\genf}_{{t}}^{\bs{\varsigma}}
   =\ol{\genf}_{{t}}^{\bs{\tau}}, \enspace 
   \mbox{ for all  }\enspace t\ne h+1,  h+2 ;
\end{aligned}
\end{equation}
while $T_i(x_j) = x_j$  for $x_j \in \{ e_j, e_{\ol{j}},  f_j, f_{\ol{j}} \}$ and {$|i-j| >1$} implies
\begin{equation}\label{classical1_2}
\begin{aligned}
   &{\gene}_{{h+1}}^{\bs{\varsigma}}
   =T_{w{'}}(e_{i})
   =T_{w{'}} T_{j} (e_{i})
   ={\gene}_{{h+2}}^{\bs{\tau}}, \quad
   {\gene}_{{h+2}}^{\bs{\varsigma}}
   =T_{w{'}}T_{{i}}(e_{j})
   =T_{w{'}}(e_{j})
   ={\gene}_{{h+1}}^{\bs{\tau}},\\
   &\ol{\gene}_{{h+1} }^{\bs{\varsigma}}
   =T_{w{'}}(e_{\ol{i}})
   =T_{w{'}} T_{j} (e_{\ol{i}})
   =\ol{\gene}_{{h+2}}^{\bs{\tau}}\quad
   \ol{\gene}_{{h+2} }^{\bs{\varsigma}}
   =T_{w{'}}T_{{i}}(e_{\ol{j}})
   =T_{w{'}}(e_{\ol{j}})
   =\ol{\gene}_{{h+1}}^{\bs{\tau}},\\   
   &{\genf}_{{h+1} }^{\bs{\varsigma}}
   =T_{w{'}}(f_{i})
   =T_{w{'}} T_{j} (f_{i})
   ={\genf}_{{h+2}}^{\bs{\tau}}, \quad
   {\genf}_{{h+2} }^{\bs{\varsigma}}
   =T_{w{'}}T_{{i}}(f_{j})
   =T_{w{'}}(f_{j})
   ={\genf}_{{h+1}}^{\bs{\tau}}, \\
   &\ol{\genf}_{{h+1} }^{\bs{\varsigma}}
   =T_{w{'}}(f_{\ol{i}})
   =T_{w{'}} T_{j} (f_{\ol{i}})
   =\ol{\genf}_{{h+2}}^{\bs{\tau}}, \quad
   \ol{\genf}_{{h+2} }^{\bs{\varsigma}}
   =T_{w{'}}T_{{i}}(f_{\ol{j}})
   =T_{w{'}}(f_{\ol{j}})
   =\ol{\genf}_{{h+1}}^{\bs{\tau}}.  
\end{aligned}
\end{equation}
Summarizing \eqref{classical1_0},  \eqref{classical1_1},  \eqref{classical1_2},
it follows
\begin{equation}\label{classical1_r}
\begin{aligned}
 {\genxe}^{\bs{\varsigma}}_{i,j} =  {\genxe}^{\bs{\tau}}_{i,j}, \quad
  {\ol{\genxe}}^{\bs{\varsigma}}_{i,j} =    {\ol{\genxe}}^{\bs{\tau}}_{i,j},
\end{aligned}
\end{equation}
for all {$1 \le i \ne j \le n$}, 
hence $\psi^-_{\bs{\varsigma}} = \psi^-_{\bs{\tau}} $ and  $\psi^+_{\bs{\varsigma}} = \psi^+_{\bs{\tau}} $.

\item[Case 2:]
Assume that $w^{\bs{\varsigma}} =w{'}s_{{i}} s_{{j}} s_{{i}} w{''}$, $w^{\bs{\tau}} =w{'} s_{{j}} s_{{i}} s_{{j}}w{''}$, 
  with $l(w{'}) = h$ and {$ | {i}  -  {j} | = 1 $}.
By the action in \eqref{action_dual}, we see that
   {$ s_{i} (\alpha_{j})= s_{j} (\alpha_{i}) $}, 
{$  s_{i}s_{j}(\alpha_{i}) =\alpha_{j}$}, 
and 
{$ s_{j} s_{i}(\alpha_{j}) =\alpha_{i}$}. 
Hence, the relationships between the roots in the classical case corresponding to the reduced expressions $w^{\bs{\varsigma}}$ and $w^{\bs{\tau}}$ are as follows:
\begin{equation}\label{classical2_0}
\begin{aligned}
&\beta_{t}^{\bs{\varsigma}}=\beta_{t}^{\bs{\tau}} , \enspace  \mbox{for all }\enspace t \ne h+1,  h+2,  h+3, \\
&\beta_{h+1}^{\bs{\varsigma}}
=w' (\alpha_{i})
=w' s_{j} s_{i}(\alpha_{j})
=\beta_{h+3}^{\bs{\tau}}, \\
&\beta_{h+2}^{\bs{\varsigma}}
=w' s_{i} (\alpha_{j})
=w'  s_{j} (\alpha_{i})
=\beta_{h+2}^{\bs{\tau}}, \\
&\beta_{h+3}^{\bs{\varsigma}}
=w' s_{i}s_{j}(\alpha_{i})
=w' (\alpha_{j})
=\beta_{h+1}^{\bs{\tau}}.
\end{aligned}
\end{equation}
On the other hand, 
 \eqref{br_relation1} implies
\begin{equation}\label{classical2_1}
\begin{aligned}
   &{\gene}_{{t} }^{\bs{\varsigma}}
   ={\gene}_{{t} }^{\bs{\tau}}, \quad 
   \ol{\gene}_{{t} }^{\bs{\varsigma}}
   =\ol{\gene}_{{t} }^{\bs{\tau}}, \quad
   {\genf}_{{t} }^{\bs{\varsigma}}
   ={\genf}_{{t}}^{\bs{\tau}}, \quad 
   \ol{\genf}_{{t} }^{\bs{\varsigma}}
   =\ol{\genf}_{{t}}^{\bs{\tau}}, \enspace 
   \mbox{if}\enspace t\ne  h+1,  h+2,  h+3,
\end{aligned}
\end{equation}
 and    Remark \ref{classicalnote2} implies
\begin{equation}\label{classical2_2}
\begin{aligned}
  &{\gene}_{{h+1} }^{\bs{\varsigma}}
  =T_{w{'}}({e}_{i})
  =T_{w{'}}T_{j}T_{i}({e}_{j})
  ={\gene}_{{h+3} }^{\bs{\tau}},\\
&\ol{\gene}_{{h+1} }^{\bs{\varsigma}}
  =T_{w{'}}(e_{\ol{i}})
  =T_{w{'}}T_{j}T_{i}(e_{\ol{j}})
  =\ol{\gene}_{{h+3} }^{\bs{\tau}},\\
  &{\genf}_{{h+1} }^{\bs{\varsigma}}
  =T_{w{'}}({f}_{i})
  =T_{w{'}}T_{j}T_{i}({f}_{j})
  ={\genf}_{{h+3} }^{\bs{\tau}},\\
    &\ol{\genf}_{{h+1} }^{\bs{\varsigma}}
  =T_{w{'}}({f}_{i})
  =T_{w{'}}T_{j}T_{i}({f}_{\ol{j}})
  = \ol{\genf}_{{h+3} }^{\bs{\tau}},\\
  &{\gene}_{{h+3} }^{\bs{\varsigma}}
  =T_{w{'}}T_{i}T_{j}({e}_{i})
  =T_{w{'}}({e}_{j})
  ={\gene}_{{h+1} }^{\bs{\tau}},\\
&   {\ol{\gene}}_{{h+3}}^{\bs{\varsigma}}
   =T_{w{'}}T_{i}T_{j}({e}_{\ol{i}})
   =T_{w{'}}({e}_{\ol{j}})
   ={\ol{\gene}}_{{h+1} }^{\bs{\tau}},\\
   &{\genf}_{{h+3} }^{\bs{\varsigma}}
  =T_{w{'}}T_{i}T_{j}({f}_{i})
  =T_{w{'}}({f}_{j})
  ={\genf}_{{h+1} }^{\bs{\tau}},\\
&   {\ol{\genf}}_{{h+3}}^{\bs{\varsigma}}
   =T_{w{'}}T_{i}T_{j}({f}_{\ol{i}})
   =T_{w{'}}({f}_{\ol{j}})
   ={\ol{\genf}}_{{h+1} }^{\bs{\tau}}.
\end{aligned}
\end{equation}
In the meanwhile, as 
$T_i(x_j) = -T_j(x_i)  $  for $x_j \in \{ e_j, e_{\ol{j}},  f_j, f_{\ol{j}} \}$ and {$|i-j| = 1$}, we have
\begin{equation}\label{classical2_3}
\begin{aligned}
 &{\gene}_{{h+2} }^{\bs{\varsigma}}
 =T_{w{'}}T_{i}({e}_{j})
  =-T_{w{'}}T_{j}({e}_{i})
  =-{\gene}_{{h+2} }^{\bs{\tau}}, \\
&{\ol{\gene}}_{{h+2} }^{\bs{\varsigma}}
   =T_{w{'}}T_{i}({e}_{\ol{j}})
  =-T_{w{'}}T_{j}({e}_{\ol{i}})
  =-{\ol{\gene}}_{{h+2} }^{\bs{\tau}},\\
 &{\genf}_{{h+2} }^{\bs{\varsigma}}
 =T_{w{'}}T_{i}({f}_{j})
 =-T_{w{'}}T_{j}({f}_{i})  
= -{\genf}_{{h+2} }^{\bs{\tau}}, \\
&{\ol{\genf}}_{{h+2} }^{\bs{\varsigma}}
   =T_{w{'}}T_{i}({f}_{\ol{j}})
   =-T_{w{'}}T_{j}({f}_{\ol{i}})
 =  -{\ol{\genf}}_{{h+2} }^{\bs{\tau}} .
\end{aligned}
\end{equation}
\end{enumerate}
Summarizing equations \eqref{classical2_0},  \eqref{classical2_1},  \eqref{classical2_2},  \eqref{classical2_3}, 
it follows
\begin{equation}\label{classical2_r}
\begin{aligned}
& {\genxe}^{\bs{\varsigma}}_{i,j} =  -{\genxe}^{\bs{\tau}}_{i,j}, \ 
  {\ol{\genxe}}^{\bs{\varsigma}}_{i,j} =    -{\ol{\genxe}}^{\bs{\tau}}_{i,j}, \ 
 {\genxe}^{\bs{\varsigma}}_{j,i} =  -{\genxe}^{\bs{\tau}}_{j,i}, \ 
  {\ol{\genxe}}^{\bs{\varsigma}}_{j,i} =    -{\ol{\genxe}}^{\bs{\tau}}_{j,i},
  \mbox{ if } {\alpha}_{i,j} = \beta_{h+2}^{\bs{\varsigma}} = \beta_{h+2}^{\bs{\tau}};
   \\
& {\genxe}^{\bs{\varsigma}}_{i,j} =  {\genxe}^{\bs{\tau}}_{i,j}, \ 
  {\ol{\genxe}}^{\bs{\varsigma}}_{i,j} =   {\ol{\genxe}}^{\bs{\tau}}_{i,j}, \ 
 {\genxe}^{\bs{\varsigma}}_{j,i} =  {\genxe}^{\bs{\tau}}_{j,i}, \ 
  {\ol{\genxe}}^{\bs{\varsigma}}_{j,i} =   {\ol{\genxe}}^{\bs{\tau}}_{j,i},
  \mbox{ otherwise} .
\end{aligned}
\end{equation}
Hence we have proved the following theorem.
\begin{thm}\label{classical_2_1}
Assume {$w^{\bs{\gamma}}$} is a  reduced expression of $w_0$, 
$1\leq i<j\leq n$, 
then 
we have 
\begin{align*}
&{\genxe}^{\bs{\gamma}}_{i,j} =r^{\bs{\gamma}}_{i,j}{\genxe}^{\bs{\sigma}}_{i,j} ,\qquad
{\ol{\genxe}}^{\bs{\gamma}}_{i,j} =r^{\bs{\gamma}}_{i,j}{\genxe}^{\bs{\sigma}}_{i,j} ,\\
&{\genxf}^{\bs{\gamma}}_{j,i} =r^{\bs{\gamma}}_{j,i}{\genxf}^{\bs{\sigma}}_{j,i} ,\qquad
{\ol{\genxf}}^{\bs{\gamma}}_{j,i}  =r^{\bs{\gamma}}_{j,i}{\genxf}^{\bs{\sigma}}_{j,i} ,
\end{align*}  
where {$w^{\bs{\sigma}}$} is  the fixed reduced expression  in \eqref{w_0_fix}, 
and  $r^{\bs{\gamma}}_{i,j},r^{\bs{\gamma}}_{j,i} \in \{\pm 1\}$.
\end{thm}

\section{The  PBW basis of the Kostant {$\ZZ$}-form $\UqnZ$}\label{PBW_0}

Referring to  \cite[(3.1)]{DW}, 
recall the elements {${\vx}_{i, j} $} and {$\ol{\vx}_{i, j} $} in {$\Uqn$} for  {$1 \le i \ne j  \le n$},
and it is trivial that 
\begin{align*}
& {\vx}_{i, j} = -{\genxe}^{\bs{\sigma}}_{i,j} , \qquad
\ol{\vx}_{i, j} =  -{\ol{\genxe}}^{\bs{\sigma}}_{i,j} . 
\end{align*}

For any {$X $}, {$m \in \NN$},
and {$\bs{k} \in \NN^n$},  set
\begin{align*}
X^{(m)} = \frac{X^m}{m!}, \qquad
{\binom{X}{m}} = \frac{X(X-1) \cdots (X-m+1)}{m!}, \qquad
{\binom{\bs{h}} {\bs{k}}} = \prod_{i=1}^{n} \binom{h_{i}}{k_i}.
\end{align*}

By \cite[Section 4]{BK2},  the Kostant {$\ZZ$}-form of $\Uqn$, denoted by $\UqnZ$, is the {$\ZZ$}-subalgebra generated by
\begin{align*}
\fsG_{\ZZ} = \left\{  \binom{\bs{h}}{\bs{k}},   h_{\ol{i}},  e^{(t)}_{j},   e_{\ol{j}},  f^{(t)}_{j},   f_{\ol{j}}
\where i\in[1,n],j\in[1,n-1], t \in \NN,  \bs{k} \in \NN^n  \right\}.
\end{align*}
 Let {$\Uz^{0}$} be the $\ZZ$-subalgebra of {$\UqnZ$} generated by {$ \binom{\bs{h}}{\bs{k}}$} and {$h_{\ol{i}} $}for $1 \le i \le n $, 
   and let {$\Uz^{+}$} (resp. {$\Uz^{-}$})be the $\ZZ$-subalgebra of {$\UqnZ$} generated by $e_j$ and $ e_{\ol{j}} $ (resp. $f_j, f_{\ol{j}} $) for $1 \le j \le n-1 $. 
   Referring to \cite[Proposition 3.3]{DW},  there is a linear space isomorphism
   $$
   \UqnZ \cong \Uz^{-} \otimes \Uz^{0} \otimes \Uz^{+}.
   $$

With the same disscussion with  \cite[Proposition 3.3]{DW}
and \cite[Proposition 6.14]{GLL},
and applying Theorem \ref{classical_2_1}, 
 we have 
\begin{thm}\label{thm_classical_pbw}
Assume {$w^{\bs{\gamma}}$} is any  reduced expression of $w_0$, 
then 
we have 
\begin{enumerate}
\item
   	 the set
$
   \{	 {\genxe}^{\bs{\gamma}}_{A^+} =\prod_{n\ge i\ge 1} \prod_{n\ge j\ge i+1} ({{\genxe}}^{\bs{\gamma}}_{i, j})^{(\SE{a}_{i, j})} (\ol{{\genxe}}^{\bs{\gamma}}_{i, j})^{\SO{a}_{i, j}}
      \where  A \in \MNZN(n)\} 
$
   	forms a $\ZZ$-basis of {$\Uz^{+}$};
\item
     the set
$
 \{   {\genxf}^{\bs{\gamma}}_{A^-} 
   =\prod_{n\ge j\ge 2} \prod_{1\le i\le j-1}({{\genxf}}^{\bs{\gamma}}_{j, i})^{(\SE{a}_{j, i})} (\ol{{\genxf}}^{\bs{\gamma}}_{j, i})^{\SO{a}_{j, i}}
     \where  A \in \MNZN(n)\} 
$
   	forms a $\ZZ$-basis of {$\Uz^{-}$};
\item
        the set
   	 $ \{ {\genxh}_{A, \bs{j}} =\prod\limits_{ i=1}^{n}  \binom{h_{i}}{j_i} h_{\ol {i}}^{a_{i,i}^{\ol{1}}} \where \bs{j} =(j_1 , \dots , j_n )\in \NN^{n} , A \in \MNZN(n) \}$
   	 forms a $\ZZ$-basis of {$\Uz^{0}$};
     \item
   	the set
   	$ \{  \bs{\frb}^{\bs{\gamma}}_{A} = {{f}}^{\bs{\gamma}}_{A} \cdot  {\genxh}_{A, \bs{j}}  \cdot {{e}}^{\bs{\gamma}}_{A}  
   	\where  \bs{j} =(j_1 , \dots , j_n )\in \NN^{n} , A \in \MNZN(n) \}$
   	   	forms a $\ZZ$-basis of {$\UqnZ$}.   
\end{enumerate}
\end{thm}

\section{The action of $\fcB$ on the quantum queer superalgebra $\Uvqn$}\label{braid_uvqn}
In \cite[Section 4]{Ol}, Olshanski introduced a quantum deformation of the  universal enveloping algebra of  $\qn$ 
using a modification of the Reshetikhin-Takhtajan-Faddeev method, denoted as {$\Uvqn$}. 
Du and Wan gives  an equivalent definition of {$\Uvqn$} in  \cite[Proposition 5.2]{DW}.
\begin{defn}\cite[Proposition 5.2]{DW}\label{defqn}
The quantum queer superalgebra $\Uvqn$ is the (Hopf) superalgebra
over $\Qv$  generated by
even generators  {${\genK}_{i}$}, {${\genK}_{i}^{-1}$},  {${\genE}_{j}$},  {${\genF}_{j}$},
and odd generators  {${\genK}_{\ol{i}}$},  {${\genE}_{\ol{j}}$}, {${\genF}_{\ol{j}}$},
with  {$ 1 \le i \le n$}, {$ 1 \le j \le n-1$}, subjecting to the following relations:
\begin{align*}
({\rm QQ1})\quad
&	{\genK}_{i} {\genK}_{i}^{-1} = {\genK}_{i}^{-1} {\genK}_{i} = 1,  \qquad
	{\genK}_{i} {\genK}_{j} = {\genK}_{j} {\genK}_{i} , \qquad
	{\genK}_{i} {\genK}_{\ol{j}} = {\genK}_{\ol{j}} {\genK}_{i}, \\
&	{\genK}_{\ol{i}} {\genK}_{\ol{j}} + {\genK}_{\ol{j}} {\genK}_{\ol{i}}
	= 2 {\delta}_{i,j} \frac{{\genK}_{i}^2 - {\genK}_{i}^{-2}}{{v}^2 - {v}^{-2}}; \\
({\rm QQ2})\quad
& 	{\genK}_{i} {\genE}_{j} = {v}^{(\bs{\ep}_i, \alpha_j)} {\genE}_{j} {\genK}_{i}, \qquad
	{\genK}_{i} {\genE}_{\ol{j}} = {v}^{(\bs{\ep}_i, \alpha_j)} {\genE}_{\ol{j}} {\genK}_{i}, \\
& 	{\genK}_{i} {\genF}_{j} = {v}^{-(\bs{\ep}_i, \alpha_j)} {\genF}_{j} {\genK}_{i}, \qquad
	{\genK}_{i} {\genF}_{\ol{j}} = {v}^{-(\bs{\ep}_i, \alpha_j)} {\genF}_{\ol{j}} {\genK}_{i}; \\
({\rm QQ3})\quad
& {\genK}_{\ol{i}} {\genE}_{i} - {v} {\genE}_{i} {\genK}_{\ol{i}} = {\genE}_{\ol{i}} {\genK}_{i}^{-1}, \qquad
	{v} {\genK}_{\ol{i}} {\genE}_{i-1} -  {\genE}_{i-1} {\genK}_{\ol{i}} = - {\genK}_{i}^{-1} {\genE}_{\ol{i-1}}, \\
& {\genK}_{\ol{i}} {\genF}_{i} - {v} {\genF}_{i} {\genK}_{\ol{i}} = - {\genF}_{\ol{i}} {\genK}_{i}, \qquad
	{v} {\genK}_{\ol{i}} {\genF}_{i-1} -  {\genF}_{i-1} {\genK}_{\ol{i}} = {\genK}_{i} {\genF}_{\ol{i-1}},\\
& {\genK}_{\ol{i}} {\genE}_{\ol{i}} + {v} {\genE}_{\ol{i}} {\genK}_{\ol{i}} = {\genE}_{i} {\genK}_{i}^{-1}, \qquad
	{v} {\genK}_{\ol{i}} {\genE}_{\ol{i-1}} +  {\genE}_{\ol{i-1}} {\genK}_{\ol{i}} =   {\genK}_{i}^{-1} {\genE}_{i-1}, \\
& {\genK}_{\ol{i}} {\genF}_{\ol{i}} + {v} {\genF}_{\ol{i}} {\genK}_{\ol{i}} =   {\genF}_{i} {\genK}_{i}, \qquad
	{v} {\genK}_{\ol{i}} {\genF}_{\ol{i-1}} +  {\genF}_{\ol{i-1}} {\genK}_{\ol{i}} = {\genK}_{i} {\genF}_{i-1}, \\
&
 {\genK}_{\ol{i}} {\genE}_{j} - {\genE}_{j} {\genK}_{\ol{i}} =  {\genK}_{\ol{i}} {\genF}_{j} - {\genF}_{j} {\genK}_{\ol{i}}
 = {\genK}_{\ol{i}} {\genE}_{\ol{j}} + {\genE}_{\ol{j}} {\genK}_{\ol{i}} =  {\genK}_{\ol{i}} {\genF}_{\ol{j}} + {\genF}_{\ol{j}} {\genK}_{\ol{i}}
	= 0 \mbox{ for } j \ne i, i-1; \\
({\rm QQ4}) \quad
& {\genE}_{i} {\genF}_{j} - {\genF}_{j} {\genE}_{i}
	= \delta_{i,j}  \frac{{\genK}_{i} {\genK}_{i+1}^{-1} - {\genK}_{i}^{-1}{\genK}_{i+1}}{{v} - {v}^{-1}}, \\
&
{\genE}_{\ol{i}} {\genF}_{\ol{j}} + {\genF}_{\ol{j}} {\genE}_{\ol{i}}
	= \delta_{i,j}  ( \frac{{\genK}_{i} {\genK}_{i+1} - {\genK}_{i}^{-1} {\genK}_{i+1}^{-1}}{{v} - {v}^{-1}}
	 + ({v} - {v}^{-1}) {\genK}_{\ol{i}} {\genK}_{\ol{i+1}} )  ,\\
& {\genE}_{i} {\genF}_{\ol{j}} - {\genF}_{\ol{j}} {\genE}_{i}
	= \delta_{i,j}  ( {\genK}_{i+1}^{-1} {\genK}_{\ol{i}}  - {\genK}_{\ol{i+1}} {\genK}_{i}^{-1} ) , \qquad
 {\genE}_{\ol{i}} {\genF}_{j} - {\genF}_{j} {\genE}_{\ol{i}}
	= \delta_{i,j}  ( {\genK}_{i+1} {\genK}_{\ol{i}}  - {\genK}_{\ol{i+1}} {\genK}_{i} ) ;\\
({\rm QQ5}) \quad
&{\genE}_{\ol{i}}^2 = -\frac{ {v} - {v}^{-1} }{{v} + {v}^{-1}} {\genE}_{i}^2, \quad
	{\genF}_{\ol{i}}^2 = \frac{ {v} - {v}^{-1} }{{v} + {v}^{-1}} {\genF}_{i}^2, \\
&
{\genE}_{i} {\genE}_{\ol{j}} - {\genE}_{\ol{j}} {\genE}_{i}
	=  {\genF}_{i} {\genF}_{\ol{j}} - {\genF}_{\ol{j}} {\genF}_{i}
	= 0   \quad \mbox{ for } |i - j| \ne 1,
\\
&
{\genE}_{i} {\genE}_{j} - {\genE}_{j} {\genE}_{i} = {\genF}_{i} {\genF}_{j} - {\genF}_{j} {\genF}_{i}
= {\genE}_{\ol{i}}{\genE}_{\ol{j}}  + {\genE}_{\ol{j}}  {\genE}_{\ol{i}}= {\genF}_{\ol{i}} {\genF}_{\ol{j}}  + {\genF}_{\ol{j}} {\genF}_{\ol{i}}
	= 0 \quad \mbox{ for }|i-j|  > 1, 	
\\
& {\genE}_{i} {\genE}_{i+1} - {v} {\genE}_{i+1} {\genE}_{i}
	= {\genE}_{\ol{i}} {\genE}_{\ol{i+1}} + {v} {\genE}_{\ol{i+1}} {\genE}_{\ol{i}}, \qquad
  {\genE}_{i} {\genE}_{\ol{i+1}} - {v} {\genE}_{\ol{i+1}} {\genE}_{i}
	= {\genE}_{\ol{i}} {\genE}_{i+1} - {v} {\genE}_{i+1} {\genE}_{\ol{i}}, \\
& {\genF}_{i} {\genF}_{i+1} - {v} {\genF}_{i+1} {\genF}_{i}
	= - ({\genF}_{\ol{i}} {\genF}_{\ol{i+1}} + {v} {\genF}_{\ol{i+1}} {\genF}_{\ol{i}}), \qquad
  {\genF}_{i} {\genF}_{\ol{i+1}}  - {v} {\genF}_{\ol{i+1}}  {\genF}_{i}
	=  {\genF}_{\ol{i}} {\genF}_{i+1} - {v} {\genF}_{i+1} {\genF}_{\ol{i}} ;\\
({\rm QQ6}) \quad
& {\genE}_{i}^2 {\genE}_{j} - ( {v} + {v}^{-1} ) {\genE}_{i} {\genE}_{j} {\genE}_{i} + {\genE}_{j}  {\genE}_{i}^2 = 0, \qquad
	{\genF}_{i}^2 {\genF}_{j} - ( {v} + {v}^{-1} ) {\genF}_{i} {\genF}_{j} {\genF}_{i} + {\genF}_{j}  {\genF}_{i}^2 = 0, \\
& {\genE}_{i}^2 {\genE}_{\ol{j}} - ( {v} + {v}^{-1} ) {\genE}_{i} {\genE}_{\ol{j}} {\genE}_{i} + {\genE}_{\ol{j}}  {\genE}_{i}^2 = 0, \qquad
	{\genF}_{i}^2 {\genF}_{\ol{j}} - ( {v} + {v}^{-1} ) {\genF}_{i} {\genF}_{\ol{j}} {\genF}_{i} + {\genF}_{\ol{j}}  {\genF}_{i}^2 = 0,  \\
& \qquad \mbox{ where } \quad |i-j| = 1.
\end{align*}
\end{defn}

\begin{rem}\label{induction_N}
By observing the relations in (QQ3)
and applying the third relation  in  (QQ4), one provides
\begin{align*}
{\genE}_{\ol{j}}
&= - {v} {\genK}_{j+1} {\genK}_{\ol{j+1}} {\genE}_{j} + {\genK}_{j+1}  {\genE}_{j}  {\genK}_{\ol{j+1}} , \\
{\genF}_{\ol{j}}
&= {v}  {\genK}_{j+1}^{-1} {\genK}_{\ol{j+1}} {\genF}_{j} - {\genK}_{j+1}^{-1}  {\genF}_{j}   {\genK}_{\ol{j+1}},\\
{\genK}_{\ol{j}}
&=
	{\genE}_{j} {\genF}_{\ol{j}}  {\genK}_{j+1}  - {\genF}_{\ol{j}} {\genE}_{j} {\genK}_{j+1}  +  {\genK}_{j}^{-1} {\genK}_{\ol{j+1}} {\genK}_{j+1}\\
&= {\genE}_{j}  {\genK}_{\ol{j+1}} {\genF}_{j}
	 - {v}^{-1} {\genE}_{j}  {\genF}_{j}  {\genK}_{\ol{j+1}}
		- {v}  {\genK}_{\ol{j+1}} {\genF}_{j}  {\genE}_{j}
		+  {\genF}_{j}  {\genK}_{\ol{j+1}} {\genE}_{j}
		+  {\genK}_{j}^{-1} {\genK}_{\ol{j+1}} {\genK}_{j+1} .
\end{align*}
Hence by induction on {$j$} in descending order  from ({$n-1$}) to $1$,
one can obtain all the other odd generators in $\Uvqn$. 
In other words,
$\Uvqn$ could be generated by $\{ {\genE}_j, {\genF}_j, {\genK}_i^{\pm1}, {\genK}_{\ol{n}}\where  1\le j \le n-1, \  1 \le i \le n \}$.
Moreover, as  in \cite{GJKKK} and \cite{DLZ},
some of the relations (QQ1)--(QQ6) could be omitted.
\end{rem}

Following \cite{CW}, there is an anti-involution 
$\Omega$ over {$\Uvqn$} with action given by 
\begin{equation}\label{omega}
\begin{aligned}
	&\Omega (v)  = v^{-1}, \quad
	\Omega(\genE_{{j}} ) =\genF_{{j}}, \quad \Omega ( \genF_{{j}} )=\genE_{{j}}, \quad \Omega ( \genK_{{i}} )=\genK_{{i}}^{-1}, \\
	&\Omega(  \genE_{\ol{j}}) =\genF_{\ol{j}}, \quad \Omega (\genF_{\ol{j}} ) =\genE_{\ol{j}}, 
		\quad \Omega (\genK_{\ol{i}}) =\genK_{\ol{i}},	
\end{aligned}
\end{equation}
where $1\le i \le n$ and $1\le j \le n-1 $. 

Let {${\boldsymbol U}_{v}^{\ol{0}}$} be the subalgebra of {$\Uvqn$} generated by the even generators {$\genK_i$}, {$\genE_j$}, {$\genF_j$}({$1 \le i \le n$}, {$1 \le j \le n-1$}) associated with some defining relations.
 To obtain the action of $\fcB$ on $\Uvqn$, we consider the action on the even part as the first step. 
 
Denote {$\wave{\genK}_{j} = \genK_{j} \genK_{j+1}^{-1}$}.
Referring to  Definition \ref{def_uvsln} and Definition \ref{defqn},
it is verified that there is an algebra monomorphism 
\begin{align*}
&   \pi_n: \Uvsln \to {\boldsymbol U}_{v}^{\ol{0}}, \\
&    K_j \mapsto \wave{\genK}_j, \qquad
 E_{j} \mapsto \genE_{j}, \qquad
 F_{j} \mapsto \genF_{j}, \qquad
\mbox{ for } j=1, \cdots , n-1.
\end{align*}
For any {$ X \in \pi_n (\Uvsln)$},
we may define the action of {$ \fc{B}$} on it  
as 
\begin{align*}
T_{i} ( X ) = \pi_n(T_i ( \pi_n^{-1} (X) ).
\end{align*}
Referring to Proposition \ref{braid_on_gn},
we have
\begin{equation}\label{action_even_kk}
\begin{aligned}
&
T_{i} ({\genK}_{i} {\genK}_{i+1}^{-1})  
=  ({\genK}_{i} {\genK}_{i+1}^{-1}) ^{-1} 
= {\genK}_{i}^{-1} {\genK}_{i+1} ,\\
&
T_{i} ({\genE}_{i}) = -{\genF}_{i} {\genK}_{i} {\genK}_{i+1}^{-1} , \quad
T_{i} ({\genF}_{i}) = - {\genK}_{i}^{-1} {\genK}_{i+1}  {\genE}_{i} ,\\
&
T_{i} ({\genK}_{i+1} {\genK}_{i+2}^{-1} )
 =  ({\genK}_{i+1} {\genK}_{i+2}^{-1})   ({\genK}_{i} {\genK}_{i+1}^{-1})   
  =   {\genK}_{i}    {\genK}_{i+2}^{-1}  ,\\
&
T_{i} ({\genK}_{i-1} {\genK}_{i}^{-1} ) 
=  ({\genK}_{i-1} {\genK}_{i}^{-1})   ({\genK}_{i} {\genK}_{i+1}^{-1})   
=  {\genK}_{i-1}   {\genK}_{i+1}^{-1} ,\\
& T_{i} ({\genE}_{j}) 
=   -  {\genE}_{i} {\genE}_{j}  +   {v}^{-1} {\genE}_{j} {\genE}_{i}, \quad
T_{i} ({\genF}_{j}) 
= -   {\genF}_{j} {\genF}_{i} +   {v} {\genF}_{i}  {\genF}_{j}  \quad \mbox{ for} \enspace |i - j| = 1, \\
&
T_{i} ({\genK}_{j} {\genK}_{j+1}^{-1} ) = {\genK}_{j} {\genK}_{j+1}^{-1}   , \quad
T_{i} ({\genE}_{j}) = {\genE}_{j}, \quad
T_{i} ({\genF}_{j}) = {\genF}_{j} 
 \quad \mbox{ for}  \enspace |i - j| > 1.
\end{aligned}
\end{equation}

Then the action of {$ \fc{B}$}  on  $\Uvsln$ will induce the action on {$\Uvqn$}.  
\begin{thm}\label{action_uvqn}
The  braid group  {$\fcB$} acts by {$\Qv$}-algebra automorphisms on {$\Uvqn$}.
The action of  {$\fcB$} on the even generators are defined as 
\begin{align*}
&
T_{j} ({\genK}^{\pm 1} _{i}  )  =   {\genK}^{\pm 1} _{s_{j}(i)}
\mbox{ for}\enspace 1 \le j \le n-1,   \quad 1 \le i \le n;\\
&
T_{i} ({\genE}_{i}) = - {\genF}_{i} {\genK}_{i} {\genK}_{i+1}^{-1} , \quad
T_{i} ({\genF}_{i}) = -  {\genK}_{i}^{-1} {\genK}_{i+1} {\genE}_{i} ;\\
& 
T_{i} ({\genE}_{j}) 
=   -  {\genE}_{i} {\genE}_{j}  +   {v}^{-1} {\genE}_{j} {\genE}_{i}, \quad
T_{i} ({\genF}_{j}) 
=  -   {\genF}_{j} {\genF}_{i} +   {v} {\genF}_{i}  {\genF}_{j}  \quad \mbox{ for} \enspace |i - j| = 1; \\
&
T_{i} ({\genE}_{j}) = {\genE}_{j} , \quad
T_{i} ({\genF}_{j}) = {\genF}_{j} 
 \quad \mbox{ for}\enspace |i - j| > 1.
\end{align*}
The action of  {$\fcB$} on the odd generators are defined as 
\begin{align*}
	&
	 T_{i} ({\genK}_{\ol{i-1}}  )  =   {\genK}_{\ol{i-1}},\quad
	T_{i} ({\genK}_{\ol{i}}  )  =   {\genK}_{\ol{i+1}} ,\quad
	T_{i} ({\genK}_{\ol{i+1}}  )  =  (v-v^{-1}) {\genK}_{\ol{i+1}}\genF_i\genE_i -(v-v^{-1})\genF_i\genE_i {\genK}_{\ol{i+1}} +\genK_{\ol{i}},  \\	
	&T_{i} ({\genE}_{\ol{i}}) = -{\genK}_{\ol{i+1}}{\genF}_{i}  {\genK}_{i} +v{\genF}_{i}{\genK}_{\ol{i+1}}  {\genK}_{i}, \quad
	T_{i} ({\genF}_{\ol{i}}) = - {\genK}_{i}^{-1} {\genE}_{i}{\genK}_{\ol{i+1}} +v^{-1} {\genK}_{\ol{i+1}}{\genK}_{i}^{-1} {\genE}_{i},\\
	& 
	 T_{i} ({\genE}_{\ol{j}}) 
	=   -  {\genE}_{i} {\genE}_{\ol{j}}  +   {v}^{-1} {\genE}_{\ol{j}} {\genE}_{i}, \quad
	T_{i} ({\genF}_{\ol{j}}) 
	= -   {\genF}_{\ol{j}} {\genF}_{i} +   {v} {\genF}_{i}  {\genF}_{\ol{j}} 
	 \quad \mbox{ for} \enspace |i - j| = 1, \\
	&
	T_{i} ({\genK}_{\ol{j}}  ) = {\genK}_{\ol{j}} , \quad
	T_{i} ({\genE}_{\ol{j}}) = {\genE}_{\ol{j}}, \quad
	T_{i} ({\genF}_{\ol{j}}) = {\genF}_{\ol{j}} 
	\quad \mbox{ for}\enspace |i - j| > 1.
\end{align*}
The inverse of $T_{i}$ is given by
\begin{align*}
&
T_{i}^{-1} ({\genK}^{\pm 1} _{i-1}  )  =   {\genK}^{\pm 1} _{i-1},\quad
T_{i}^{-1} ({\genK}^{\pm 1} _{i}  )  =   {\genK}^{\pm 1} _{i+1},\quad
T_{i} ^{-1}({\genK}^{\pm 1} _{i+1}  )  =   {\genK}^{\pm 1} _{i}  , \\
&
T_{i} ^{-1}({\genE}_{i}) =  - {\genK}_{i+1}{\genK}_{i}^{-1}{\genF}_{i}  , \quad
T_{i} ^{-1}({\genF}_{i}) = -   {\genE}_{i} {\genK}_{i} {\genK}_{i+1}^{-1},\\
& 
T_{i} ^{-1}({\genE}_{j}) 
=   -  {\genE}_{j}{\genE}_{i}   +   {v}^{-1}  {\genE}_{i}{\genE}_{j}, \quad
T_{i}^{-1} ({\genF}_{j}) 
= -    {\genF}_{i} {\genF}_{j}+   {v}   {\genF}_{j} {\genF}_{i} \quad \mbox{ for} \enspace |i - j| = 1, \\
&
T_{i}^{-1} ({\genK}^{\pm 1} _{j}  ) = {\genK}^{\pm 1} _{j} , \quad
T_{i}^{-1} ({\genE}_{j}) = {\genE}_{j}, \quad
T_{i}^{-1} ({\genF}_{j}) = {\genF}_{j} 
\quad \mbox{ for} \enspace |i - j| > 1;\\
 &
	T_{i}^{-1} ({\genK}_{\ol{i-1}}  )  =   {\genK}_{\ol{i-1}},\quad
	T_{i}^{-1} ({\genK}_{\ol{i+1}}  )  =   {\genK}_{\ol{i}} ,\\
	&
	T_{i}^{-1}  ({\genK}_{\ol{i}}  )=(v-v^{-1}) \genE_i \genF_i \genK_{\ol{i}} -(v-v^{-1}) \genE_i \genF_i \genK_{\ol{i}} +{\genK}_{\ol{i+1}}, \\	
	&T_{i}^{-1}  ({\genE}_{\ol{i}}) = -{\genK}_{{i+1}}{\genF}_{i}  {\genK}_{\ol{i}} +v{\genK}_{{i+1}}  {\genK}_{\ol{i}}{\genF}_{i}, \quad
	T_{i}^{-1}  ({\genF}_{\ol{i}}) = -{\genK}_{\ol{i}}{\genE}_{i} {\genK}_{i+1}^{-1}  +v^{-1} {\genE}_{i}{\genK}_{\ol{i}}{\genK}_{i+1}^{-1} ,\\	
	& T_{i}^{-1} ({\genE}_{\ol{j}}) 
	=   -  {\genE}_{j} {\genE}_{\ol{i}}  +   {v}^{-1} {\genE}_{\ol{i}} {\genE}_{j}, \quad
	T_{i}^{-1} ({\genF}_{\ol{j}}) 
	= -   {\genF}_{{i}} {\genF}_{\ol{j}} +   {v} {\genF}_{\ol{j}}  {\genF}_{{i}}  \quad \mbox{ for} \enspace |i - j| = 1, \\
	&
	T_{i}^{-1} ({\genK}_{\ol{j}}  ) = {\genK}_{\ol{j}} , \quad
	T_{i}^{-1} ({\genE}_{\ol{j}}) = {\genE}_{\ol{j}}, \quad
	T_{i}^{-1} ({\genF}_{\ol{j}}) = {\genF}_{\ol{j}} 
	\quad \mbox{ for}  \enspace |i - j| > 1.
\end{align*}
\end{thm}
We will prove this theorem in several parts.

\textbf{Proof of the action on the even generators:}
It is trivial the map {$\pi_n$} holds under the action of  {$T_j$}.
We only need to verify that the action of {$T_j$} on {$\genK_i$} 
satisfy the equations in \eqref{action_even_kk} 
and the definition of braid group $\fc{B}$, for all {$1\leq j \leq  n-1$} and {$1\leq i \leq n$}.

By  {$T_{i} ({\genK}_{i}  ) = {\genK}_{i+1}  $}, {$T_{i+1} ({\genK}_{i}  ) = {\genK}_{i}  $}, 
{$T_{i} ({\genK}_{i+1}  ) =   {\genK}_{i}  $} 
and {$T_{i} (  {\genK}_{i+2} ) =   {\genK}_{i+2}  $},
we have 
\begin{align*}
&
T_{i} ({\genK}_{i} {\genK}_{i+1}^{-1})  
= T_{i} ({\genK}_{i}) T_{i} ( {\genK}_{i+1}^{-1})  
= {\genK}_{i}^{-1} {\genK}_{i+1} ,\\
&
T_{i} ({\genK}_{i+1} {\genK}_{i+2}^{-1} )
=T_{i} ({\genK}_{i+1}) T_{i} ( {\genK}_{i+2}^{-1})  
= {\genK}_{i} {\genK}_{i+2}^{-1} , \\
&
T_{i} ({\genK}_{i-1} {\genK}_{i}^{-1} )
=T_{i} ({\genK}_{i-1}) T_{i} ( {\genK}_{i}^{-1})  
= {\genK}_{i-1} {\genK}_{i+1}^{-1} ,\\
&
T_{i} ({\genK}_{j} {\genK}_{j+1}^{-1} )
=T_{i}({\genK}_{j})T_{i}({\genK}_{j+1}^{-1})
={\genK}_{j} {\genK}_{j+1}^{-1} 
\quad \mbox{for } |i-j|>1.
\end{align*}
And direct calculation leads
\begin{align*}
  	&T_iT_{i+1}T_i({\genK}_{i-1})=T_iT_{i+1}({\genK}_{i-1})=T_i({\genK}_{i-1})={\genK}_{i-1},\\
  	&T_{i+1}T_{i}T_{i+1}({\genK}_{i-1})=T_{i+1}T_{i}({\genK}_{i-1})=T_{i+1}({\genK}_{i-1})={\genK}_{i-1};\\
  	&T_iT_{i+1}T_i({\genK}_i)= T_iT_{i+1}({\genK}_{i+1})=T_i({\genK}_{i+2})={\genK}_{i+2} ,\\                       
  	&T_{i+1}T_{i}T_{i+1}({\genK}_i)=T_{i+1}T_{i}({\genK}_i) =T_{i+1}({\genK}_{i+1})={\genK}_{i+2};
  	\\
  	&T_iT_{i+1}T_i({\genK}_{i+1}) =T_iT_{i+1}({\genK}_i)=T_i({\genK}_i)= {\genK}_{i+1},\\
  	&T_{i+1}T_{i}T_{i+1}({\genK}_{i+1})=T_{i+1}T_{i}({\genK}_{i+2})= T_{i+1}({\genK}_{i+2}) = {\genK}_{i+1};
  	\\
  	&T_iT_{i+1}T_i({\genK}_{i+2}) =T_iT_{i+1}({\genK}_{i+2})=T_i({\genK}_{i+1})={\genK}_i,\\
  	&T_{i+1}T_{i}T_{i+1}({\genK}_{i+2})=T_{i+1}T_{i}({\genK}_{i+1})=T_{i+1}({\genK}_i)={\genK}_i.
\end{align*}  	
Thus, the composite action $T_i T_{i+1} T_i $ and $T_{i+1} T_i T_{i+1}$ coincide on the generators $ \genK_{j}$. 
Similarly, we see that the composite action $T_iT_j$ and $T_jT_i$ are equal, where $|i-j|>1$.
Thus, the proposition is completely proven.

\textbf{Proof of the action on the odd generators:}
To prove the action on the odd generators, we need the following two lemmas.
\begin{lem}\label{verify1}
For any 
 {$X \in \{ \genK_{\ol{i}},  \genE_{\ol{j}}, \genF_{\ol{j}} \where 1\le i \le n, \ 1\le j \le n-1\}$}, we have
\begin{align*}
& T_{j} T_{j+1} T_{j} (X)= T_{j+1} T_{j} T_{j+1} (X),\\
&T_{i} T_j (X) = T_j T_{i} (X) \quad \mbox{where } |i-j|>1.
\end{align*}
In other word, $T_{j} (1 \le j \le n-1)$ satisfies the relations in \eqref{T_relation}.
\end{lem}

\begin{proof}	
For $X= \genK_{\ol{j}} $, we see that 
\begin{align*}
  	&T_{j}T_{j+1}T_{j}({\genK}_{\ol{j-1}})=T_{j}T_{j+1}({\genK}_{\ol{j-1}})={\genK}_{\ol{j-1}},\\
  	&T_{j+1}T_{j}T_{j+1}({\genK}_{\ol{j-1}})=T_{j+1}T_{j}({\genK}_{\ol{j-1}})=T_{j+1}({\genK}_{\ol{j-1}})={\genK}_{\ol{j-1}};\\
  	&T_{j}T_{j+1}T_{j}({\genK}_{\ol{j}})=T_{j}T_{j+1}({\genK}_{\ol{j+1}})=T_{j}({\genK}_{\ol{j+2}})={\genK}_{\ol{j+2}} ,\\                       
  	&T_{j+1}T_{j}T_{j+1}({\genK}_{\ol{j}})=T_{j+1}T_{j}({\genK}_{\ol{j}})=T_{j+1}({\genK}_{\ol{j+1}})={\genK}_{\ol{j+2}} ;\\
  	&T_{j}T_{j+1}T_{j}({\genK}_{\ol{j+1}})
  	=T_{j}T_{j+1}((v-v^{-1})\genK_{\ol{j+1}}\genF_{j}\genE_{j}-(v-v^{-1})\genF_{j}\genE_{j}\genK_{\ol{j+1}}+\genK_{\ol{j}})\\
  	&\qquad\qquad\qquad=(v-v^{-1})\genK_{\ol{j+2}}\genF_{j+1}\genE_{j+1}-(v-v^{-1})\genF_{j+1}\genE_{j+1}\genK_{\ol{j+2}}+\genK_{\ol{j+1}},\\	
  	&T_{j+1}T_{j}T_{j+1}({\genK}_{\ol{j+1}})
  	=T_{j+1}T_{j}({\genK}_{\ol{j+2}})
  	=(v-v^{-1})\genK_{\ol{j+2}}\genF_{j+1}\genE_{j+1}-(v-v^{-1})\genF_{j+1}\genE_{j+1}\genK_{\ol{j+2}}+\genK_{\ol{j+1}};\\
  	  &T_{j}T_{j+1}T_{j}({\genK}_{\ol{j+2}})=T_{j}T_{j+1}({\genK}_{\ol{j+2}})\\
  	&\qquad\qquad\qquad=T_{j}((v-v^{-1})\genK_{\ol{j+2}}\genF_{j+1}\genE_{j+1}-(v-v^{-1})\genF_{j+1}\genE_{j+1}\genK_{\ol{j+2}}+\genK_{\ol{j+1}})\\
  	&\qquad\qquad\qquad=(v-v^{-1})T_{j}(\genK_{\ol{j+2}})T_{j}(\genE_{j+1})-(v-v^{-1})T_{j}(\genF_{j+1})T_{j}(\genE_{j+1})T_{j}(\genK_{\ol{j+2}})+T_{j}(\genK_{\ol{j+1}})\\
  	&\qquad\qquad\qquad=(v-v^{-1})\genK_{\ol{j+2}}(-\genF_{j+1}\genF_{j}+v\genF_{j}\genF_{j+1})(-\genE_{j}\genE_{j+1}+v^{-1}\genE_{j+1}\genE_{j})\\&\qquad\qquad\qquad\quad-(v-v^{-1})(-\genF_{j+1}\genF_{j}+v\genF_{j}\genF_{j+1})(-\genE_{j}\genE_{j+1}+v^{-1}\genE_{j+1}\genE_{j})\genK_{\ol{j+2}}\\&\qquad\qquad\qquad\quad+(v-v^{-1})\genK_{\ol{j+1}}\genF_{j}\genE_{j}-(v-v^{-1})\genF_{j}\genE_{j}\genK_{\ol{j+1}}+\genK_{\ol{j}}\\	
  	&\qquad\qquad\qquad=(v-v^{-1})\genK_{\ol{j+2}}\genF_{j+1}\genF_{j}\genE_{j}\genE_{j+1}-v^{-1}(v-v^{-1})\genK_{\ol{j+2}}\genF_{j+1}\genF_{j}\genE_{j+1}\genE_{j}\\&\qquad\qquad\qquad\qquad
  	-v(v-v^{-1})\genK_{\ol{j+2}}\genF_{j}\genF_{j+1}\genE_{j}\genE_{j+1}+(v-v^{-1})\genK_{\ol{j+2}}\genF_{j}\genF_{j+1}\genE_{j+1}\genE_{j}\\&\qquad\qquad\qquad\qquad
  	-(v-v^{-1})\genF_{j+1}\genF_{j}\genE_{j}\genE_{j+1}\genK_{\ol{j+2}}
  	+v^{-1}(v-v^{-1})\genF_{j+1}\genF_{j}\genE_{j+1}\genE_{j}\genK_{\ol{j+2}}\\&\qquad\qquad\qquad\qquad
  	+v(v-v^{-1})\genF_{j}\genF_{j+1}\genE_{j}\genE_{j+1}\genK_{\ol{j+2}}
  	-(v-v^{-1})\genF_{j}\genF_{j+1}\genE_{j+1}\genE_{j}\genK_{\ol{j+2}}\\&\qquad\qquad\qquad\qquad
  	+(v-v^{-1})\genK_{\ol{j+1}}\genF_{j}\genE_{j}
  	-(v-v^{-1})\genF_{j}\genE_{j}\genK_{\ol{j+1}}
  	+\genK_{\ol{j}},
  	\\
  	&T_{j+1}T_{j}T_{j+1}({\genK}_{\ol{j+2}})=T_{j+1}T_{j}((v-v^{-1})\genK_{\ol{j+2}}\genF_{j+1}\genE_{j+1}-(v-v^{-1})\genF_{j+1}\genE_{j+1}\genK_{\ol{j+2}}+\genK_{{j+1}})\\
  	&\qquad\qquad\qquad\quad=(v-v^{-1})T_{j+1}(\genK_{\ol{j+2}})T_{j+1}T_{j+1}^{-1}(\genF_{j})T_{j+1}T_{j+1}^{-1}(\genE_{j})\\&\qquad\qquad\qquad\qquad-(v-v^{-1})T_{j+1}T_{j+1}^{-1}(\genF_{j})T_{j+1}T_{j+1}^{-1}(\genE_{j})T_{j+1}(\genK_{\ol{j+2}})\\&\qquad\qquad\qquad\qquad
  	+T_{j+1}((v-v^{-1})\genK_{\ol{j+1}}\genF_{j}\genE_{j}-(v-v^{-1})\genF_{j}\genE_{j}\genK_{\ol{j+1}}+\genK_{\ol{j}})\\
  	&\qquad\qquad\qquad\quad=(v-v^{-1})((v-v^{-1})\genK_{\ol{j+2}}\genF_{j+1}\genE_{j+1}-(v-v^{-1})\genF_{j+1}\genE_{j+1}\genK_{\ol{j+2}}+\genK_{{j+1}})\genF_{j}\genE_{j}\\&\qquad\qquad\qquad\qquad-(v-v^{-1})\genF_{j}\genE_{j}((v-v^{-1})\genK_{\ol{j+2}}\genF_{j+1}\genE_{j+1}-(v-v^{-1})\genF_{j+1}\genE_{j+1}\genK_{\ol{j+2}}+\genK_{\ol{j+1}})\\&\qquad\qquad\qquad\qquad
  	+(v-v^{-1})\genK_{\ol{j+2}}(-\genF_{j}\genF_{j+1}+v\genF_{j+1}\genF_{j})(-\genE_{j+1}\genE_{j}+v^{-1}\genE_{j}\genE_{j+1})\\&\qquad\qquad\qquad\qquad-(v-v^{-1})(-\genF_{j}\genF_{j+1}+v\genF_{j+1}\genF_{j})(-\genE_{j+1}\genE_{j}+v^{-1}\genE_{j}\genE_{j+1})\genK_{\ol{j+2}}+\genK_{\ol{j}}\\			
  	&\qquad\qquad\qquad\quad=(v-v^{-1})\genK_{\ol{j+2}}\genF_{j+1}\genF_{j}\genE_{j}\genE_{j+1}-v^{-1}(v-v^{-1})\genK_{\ol{j+2}}\genF_{j+1}\genF_{j}\genE_{j+1}\genE_{j}\\&\qquad\qquad\qquad\qquad
  	-v(v-v^{-1})\genK_{\ol{j+2}}\genF_{j}\genF_{j+1}\genE_{j}\genE_{j+1}+(v-v^{-1})\genK_{\ol{j+2}}\genF_{j}\genF_{j+1}\genE_{j+1}\genE_{j}\\&\qquad\qquad\qquad\qquad
  	-(v-v^{-1})\genF_{j+1}\genF_{j}\genE_{j}\genE_{j+1}\genK_{\ol{j+2}}
  	+v^{-1}(v-v^{-1})\genF_{j+1}\genF_{j}\genE_{j+1}\genE_{j}\genK_{\ol{j+2}}\\&\qquad\qquad\qquad\qquad
  	+v(v-v^{-1})\genF_{j}\genF_{j+1}\genE_{j}\genE_{j+1}\genK_{\ol{j+2}}
  	-(v-v^{-1})\genF_{j}\genF_{j+1}\genE_{j+1}\genE_{j}\genK_{\ol{j+2}}\\&\qquad\qquad\qquad\qquad
  	+(v-v^{-1})\genK_{\ol{j+1}}\genF_{j}\genE_{j}
  	-(v-v^{-1})\genF_{j}\genE_{j}\genK_{\ol{j+1}}
  	+\genK_{\ol{j}}.	
\end{align*}
  Thus, the compositions $T_{j} T_{j+1} T_{j} $ and $T_{j+1} T_{j} T_{j+1}$ coincide on the generators $ \genK_{\ol{j}}$. 
   Similarly, we also see that  $$T_{j}T_j(\genK_{\ol{j}})= T_jT_{j}(\genK_{\ol{j}})\quad \mbox{ where }|i-j|>1.$$
  	
For $X= \genE_{\ol{j}} $ with admissible $j$, we have   
\begin{align*}
  	&T_{j}T_{j+1}T_{j}({\genE}_{\ol{j-1}})
  	=T_{j}T_{j+1}(-\genE_{j}\genE_{\ol{j-1}}+v^{-1}\genE_{\ol{j-1}}\genE_{j})\\&\qquad\qquad\qquad
  	=-\genE_{j+1}T_{j}(\genE_{\ol{j-1}})+v^{-1}T_{j}(\genE_{\ol{j-1}})\genE_{j+1}\\&\qquad\qquad\qquad
  	=-\genE_{j+1}(-\genE_{j}\genE_{\ol{j-1}}+v^{-1}\genE_{\ol{j-1}}\genE_{j})+v^{-1}(-\genE_{j}\genE_{\ol{j-1}}+v^{-1}\genE_{\ol{j-1}}\genE_{j})\genE_{j+1}\\&\qquad\qquad\qquad
  	={\genE}_{{j+1}}{\genE}_{j}{\genE}_{\ol{j-1}}-v^{-1}{\genE}_{{j+1}}{\genE}_{\ol{j-1}}{\genE}_{j}-v^{-1}{\genE}_{j}{\genE}_{\ol{j-1}}{\genE}_{{j+1}}+v^{-2}{\genE}_{\ol{j-1}}{\genE}_{j}{\genE}_{{j+1}},\\           
  	&T_{j+1}T_{j}T_{j+1}({\genE}_{\ol{j-1}})
  	=T_{j+1}T_{j}({\genE}_{\ol{j-1}})\\&\qquad\qquad\qquad\quad
  	=T_{j+1}(-\genE_{j}\genE_{\ol{j-1}}+v^{-1}\genE_{j-1}\genE_{j})\\&\qquad\qquad\qquad\quad
  	=-(-\genE_{j+1}\genE_{j}+v^{-1}\genE_{j}\genE_{j+1})\genE_{\ol{j-1}}+v^{-1}\genE_{\ol{j-1}}(-\genE_{j+1}\genE_{j}+v^{-1}\genE_{j}\genE_{j+1})\\&\qquad\qquad\qquad	\quad
  	={\genE}_{{j+1}}{\genE}_{j}{\genE}_{\ol{j-1}}-v^{-1}{\genE}_{j}{\genE}_{{j+1}}{\genE}_{\ol{j-1}}-v^{-1}{\genE}_{\ol{j-1}}{\genE}_{{j+1}}{\genE}_{j}+v^{-2}{\genE}_{\ol{j-1}}{\genE}_{j}{\genE}_{{j+1}};\\ 
  	&T_{j}T_{j+1}T_{j}({\genE}_{\ol{j}})=T_{j}T_{j+1}(-\genK_{\ol{j+1}}\genF_{j}\genK_{j}+v\genF_{j}\genK_{\ol{j+1}}\genK_{j})=-\genK_{\ol{j+2}}\genF_{j+1}\genK_{j+1}+v\genF_{j+1}\genK_{\ol{j+2}}\genK_{j+1},\\
  	&T_{j+1}T_{j}T_{j+1}({\genE}_{\ol{j}})=T_{j+1}(\genE_{\ol{j+1}})=-\genK_{\ol{j+2}}\genF_{j+1}\genK_{j+1}+v\genF_{j+1}\genK_{\ol{j+2}}\genK_{j+1} ;\\
  &T_{j}T_{j+1}T_{j}({\genE}_{\ol{j+1}})=T_{j}T_{j+1}(-\genE_{j}\genE_{\ol{j+1}}+v^{-1}\genE_{\ol{j+1}}\genE_{j})\\
  &\qquad\qquad\qquad=-\genE_{j+1}T_{j}T_{j+1}(\genE_{\ol{j+1}})+v^{-1}T_{j}T_{j+1}(\genE_{\ol{j+1}})\genE_{j+1}\\
  &\qquad\qquad\qquad=-(v-v^{-1})\genK_{\ol{j+2}}\genF_{j+1}\genE_{j+1}\genF_{j}\genK_{j}+(v-v^{-1})\genF_{j+1}\genE_{j+1}\genK_{\ol{j+2}}\genF_{j}\genK_{j}\\&\qquad\qquad\qquad\qquad+v(v-v^{-1})\genF_{j}\genK_{\ol{j+2}}\genF_{j+1}\genE_{j+1}\genK_{j}-v(v-v^{-1})\genF_{j}\genF_{j+1}\genE_{j+1}K_{\ol{j+2}}\genK_{j}\\&\qquad\qquad\qquad\qquad-\genK_{\ol{j+1}}\genF_{j}\genK_{j}+v\genF_{j}\genK_{\ol{j+1}}\genK_{j}	,
  \\
  &T_{j+1}T_{j}T_{j+1}({\genE}_{\ol{j+1}})
  =T_{j+1}T_{j}(-\genK_{\ol{j+2}}\genF_{j+1}\genK_{j+1}+v\genF_{j+1}\genK_{\ol{j+2}}\genK_{j+1})\\
  &\qquad\qquad\qquad\enspace=-T_{j+1}(\genK_{\ol{j+2}})\genF_{j}\genK_{j}+v\genF_{j}T_{j+1}(\genK_{\ol{j+2}})\genK_{j}\\
  &\qquad\qquad\qquad\enspace=-[(v-v^{-1})\genK_{\ol{j+2}}\genF_{j+1}\genE_{j+1}-(v-v^{-1})\genF_{j+1}\genE_{j+1}\genK_{\ol{j+2}}+\genK_{\ol{j+1}}]\genF_{j}\genK_{j}\\&\qquad\qquad\qquad\qquad\enspace+v\genF_{j}[(v-v^{-1})\genK_{\ol{j+2}}\genF_{j+1}\genE_{j+1}-(v-v^{-1})\genF_{j+1}\genE_{j+1}\genK_{\ol{j+2}}+\genK_{\ol{j+1}}]\genK_{j}\\
  &\qquad\qquad\qquad\enspace=-(v-v^{-1})\genK_{\ol{j+2}}\genF_{j+1}\genE_{j+1}\genF_{j}\genK_{j}+(v-v^{-1})\genF_{j+1}\genE_{j+1}\genK_{\ol{j+2}}\genF_{j}\genK_{j}\\&\qquad\qquad\qquad\qquad\enspace+v(v-v^{-1})\genF_{j}\genK_{\ol{j+2}}\genF_{j+1}\genE_{j+1}\genK_{j}-v(v-v^{-1})\genF_{j}\genF_{j+1}\genE_{j+1}\genK_{\ol{j+2}}\genK_{j}\\&\qquad\qquad\qquad\qquad\enspace-\genK_{\ol{j+1}}\genF_{j}\genK_{j}+v\genF_{j}\genK_{\ol{j+1}}\genK_{j};
  \\  	
  &T_{j}T_{j+1}T_{j}({\genE}_{\ol{j+2}})
  =T_{j}T_{j+1}({\genE}_{\ol{j+2}})\\&\qquad\qquad\qquad
  =T_{j}(-{\genE}_{j+1}{\genE}_{\ol{j+2}}+v^{-1}{\genE}_{\ol{j+2}}{\genE}_{j+1})\\&\qquad\qquad\qquad
  =-(-{\genE}_{j}{\genE}_{{j+1}}+v^{-1}{\genE}_{{j+1}}{\genE}_{j}){\genE}_{\ol{j+2}}+v^{-1}{\genE}_{\ol{j+2}}(-{\genE}_{j}{\genE}_{{j+1}}+v^{-1}{\genE}_{{j+1}}{\genE}_{j})\\&\qquad\qquad\qquad
  ={\genE}_{j}{\genE}_{j+1}{\genE}_{\ol{j+2}}-v^{-1}{\genE}_{j+1}{\genE}_{{j}}{\genE}_{\ol{j+2}}-v^{-1}{\genE}_{\ol{j+2}}{\genE}_{{j}}{\genE}_{j+1}+v^{-2}{\genE}_{\ol{j+2}}{\genE}_{j+1}{\genE}_{{j}},\\
  &T_{j+1}T_{j}T_{j+1}({\genE}_{\ol{j+2}})
  =T_{j+1}T_{j}(-{\genE}_{j+1}{\genE}_{{j+2}}+v^{-1}{\genE}_{{j+2}}{\genE}_{j+1})\\&\qquad\qquad\qquad\quad
  =-{\genE}_{j}T_{j+1}({\genE}_{{j+2}})+v^{-1}T_{j+1}({\genE}_{{j+2}}){\genE}_{j}\\&\qquad\qquad\qquad\quad
  =-{\genE}_{j}(-{\genE}_{j+1}{\genE}_{{j+2}}+v^{-1}{\genE}_{{j+2}}{\genE}_{j+1})+v^{-1}(-{\genE}_{j+1}{\genE}_{{j+2}}+v^{-1}{\genE}_{{j+2}}{\genE}_{j+1}){\genE}_{j}\\&\qquad\qquad\qquad\quad	
  ={\genE}_{{j}}{\genE}_{j+1}{\genE}_{\ol{j+2}}-v^{-1}{\genE}_{{j}}{\genE}_{\ol{j+2}}{\genE}_{j+1}-v^{-1}{\genE}_{j+1}{\genE}_{\ol{j+2}}{\genE}_{{j}}+v^{-2}{\genE}_{\ol{j+2}}{\genE}_{j+1}{\genE}_{{j}}. 	
\end{align*}
  	Thus, the compositions $T_{j} T_{j+1} T_{j} $ and $T_{j+1} T_{j} T_{j+1}$ coincide on the generators $ \genE_{\ol{j}}$. Similarly, we see that 
$$
   T_iT_j(\genE_{\ol{j}})=T_jT_i(\genE_{\ol{j}})
   \mbox{ for} \enspace |i-j|>1. 
$$
   From the previous formulas and applying $\Omega$, we deduce that they also coincide on the generators $\genF_j$.
\end{proof}

\begin{lem}\label{verify2}  
 For any $Y\in \fc{B}$, 
 $Y$ holds  the equations  in (QQ1)-(QQ6).
\end{lem}
\begin{proof}
According to Proposition \ref{braid_on_gn}, 
for any $Y\in \fc{B}$, the elements in $\{ Y({\genE}_j),  Y({\genF}_j),  Y({\genK}_i ^{\pm1})\where  1\le j \le n-1,  \  1 \le i \le n \}$ satisfy all the generating relations in (QQ1)-(QQ6) composed solely of even generators. 
Therefore, we need to verify that the elements in $\{ Y({\genE}_j)$,  $Y({\genF}_j)$,  $Y({\genK}_i ^{\pm1})$, $Y({\genE}_{\ol{j}})$,  $Y({\genF}_{\ol{j}})$, $Y({\genK}_{\ol{i}})\where  1\le j \le n-1,  \  1 \le i \le n \}$ satisfy the remaining relations in (QQ1)-(QQ6). 

Due to the group structure of  $\fc{B}$ and the group action {on $\Uvqn$}, 
we only need to verify that when $Y=T_t (t \in \{i-1, i, i+1, j-1, j ,j+1\})$, the elements in the above set satisfy the remaining generating relations in (QQ1)-(QQ6).
 	
 	For the  relations in (QQ1), when $Y=T_j$, we see that 
\begin{align*}
        Y({\genK}_{i})Y({\genK}_{\ol{j}})
        = {\genK}_{s_j (i)} {\genK}_{\ol{j+1}}
        ={\genK}_{\ol{j+1}} {\genK}_{s_j (i)}
        =Y({\genK}_{\ol{j}})Y({\genK}_{i}) .  
\end{align*}
      As for the case {$Y=T_i$}, we have 
\begin{align*}
 	T_i ({\genK}_{\ol{i}})T_i ( {\genK}_{\ol{i}} )+ T_i ( {\genK}_{\ol{i}}) T_i ( {\genK}_{\ol{i}})
 	&=2 {\genK}_{\ol{i+1}}  {\genK}_{\ol{i+1}} \\
 	&= 2   \frac{{\genK}_{i+1}^2 - {\genK}_{i+1}^{-2}}{{v}^2 - {v}^{-2}}\\
 	&= T_i(2  \frac{{\genK}_{i}^2 - {\genK}_{i}^{-2}}{{v}^2 - {v}^{-2}})\\
       &=2  \frac{{T_i ({\genK}_{i})}^2 - {T_i ({\genK}_{i})}^{-2}}{{v}^2 - {v}^{-2}}.
\end{align*} 
 	Similarly, the other cases with equations in (QQ1) can be proved by direct calculation.
 	 	
 	For the  relations in (QQ2),  {$Y=T_i, T_{i+1}$} leads 
\begin{align*}
 	&T_i(\genK_i)T_i (\genE_{\ol{i}})
  =\genK_{i+1}(-\genK_{\ol{i+1}}\genF_i\genK_i +v\genF_i\genK_{\ol{i+1}}\genK_i)
  =vT_i(\genE_{\ol{i}})T_i (\genK_i), \\
 	&T_{i+1}(\genK_i)T_{i+1} (\genF_{\ol{i}})
  =\genK_i(-\genF_{\ol{i}}\genF_{i+1}+v\genF_{i+1}\genF_{\ol{i}})
  =
  =v^{-1}T_{i+1}(\genF_{\ol{i}})T_{i+1} (\genK_i).
\end{align*} 	
 	In the same way , the other cases with equations in (QQ2) can be checked. 
 	
 	For the  relations in (QQ3), we only give the proof of the second one and the fifth one. The others are analogous to them. If $Y=T_i$, then
\begin{align*}
 	vT_i(\genK_{\ol{i}})T_i(\genE_{i-1})-T_i(\genE_{i-1})T_i(\genK_{\ol{i}})
 	&=v\genK_{\ol{i+1}}(-\genE_i\genE_{i-1}+v^{-1}\genE_{i-1}\genE_i)-(-\genE_i\genE_{i-1}+v^{-1}\genE_{i-1}\genE_i)\genK_{\ol{i+1}}\\
 	&=-v\genK_{\ol{i+1}}\genE_i\genE_{i-1}
 	+\genK_{\ol{i+1}}\genE_{i-1}\genE_i
 	+\genE_i\genE_{i-1}\genK_{\ol{i+1}}
 	-v^{-1}\genE_{i-1}\genE_i\genK_{\ol{i+1}}\\
 	&=(-v\genK_{\ol{i+1}}\genE_i+\genE_i\genK_{\ol{i+1}})\genE_{i-1}+\genE_{i-1}(\genK_{\ol{i+1}}\genE_i-v^{-1}\genE_i\genK_{\ol{i+1}})\\
 	&=-\genK_{i+1}^{-1}(-\genE_{\ol{i}}\genE_{i-1}+v^{-1}\genE_{i-1}\genE_{\ol{i}})\\
 	&=-\genK_{i+1}^{-1}(-\genE_i\genE_{\ol{i-1}}+v^{-1}\genE_{\ol{i-1}}\genE_i)\\
 	&=-T_i(\genK_i^{-1})T_i (\genE_{\ol{i-1}}), 
 	\\ 
 	T_i(\genK_{\ol{i}})T_i(\genE_{\ol{i}})+vT_i(\genE_{\ol{i}})T_i(\genK_{\ol{i}})
 	&=\genK_{\ol{i+1}}(-\genK_{\ol{i+1}}\genF_i\genK_i+v\genF_i\genK_{\ol{i+1}}\genK_i)+v(-\genK_{\ol{i+1}}\genF_i\genK_i+v\genF_i\genK_{\ol{i+1}}\genK_i)\genK_{\ol{i+1}}\\
 	&=-\genK_{\ol{i+1}}^{2}\genF_i\genK_i+v^{2}\genF_i\genK_{\ol{i+1}}^{2}\genK_i\\
 	&=(-\genF_i\genK_i\genK_{i+1}^{-1})\genK_{i+1}^{-1}\\
 	&=T_i(\genE_i)T_i (\genK_i^{-1}).
\end{align*}

 	For the second  relation in (QQ4), we only give the proof for the case $i=j$. And the other cases with relations in (QQ4) are similar to prove and we omit it.
 	For the case $Y=T_i$, we see that 
\begin{align*} T_i(\genE_{\ol{i}}\genF_{\ol{i}}+\genF_{\ol{i}}\genE_{\ol{i}})
&=T_i(\genE_{\ol{i}})T_i(\genF_{\ol{i}})+T_i(\genF_{\ol{i}})T_i(\genE_{\ol{i}})\\
&=(-\genK_{\ol{i+1}}\genF_i\genK_i+v\genF_i\genK_{\ol{i+1}}\genK_i)(-\genK_i^{-1}\genE_i\genK_{\ol{i+1}}+v^{-1}\genK_{\ol{i+1}}\genK_i^{-1}\genE_i)\\&\quad+(-\genK_i^{-1}\genE_i\genK_{\ol{i+1}}+v^{-1}\genK_{\ol{i+1}}\genK_i^{-1}\genE_i)(-\genK_{\ol{i+1}}\genF_i\genK_i+v\genF_i\genK_{\ol{i+1}}\genK_i)\\
 &=\genK_{\ol{i+1}}\genF_i\genE_i\genK_{\ol{i+1}}-v^{-1}\genK_{\ol{i+1}}\genF_i\genK_{\ol{i+1}}\genE_i-v\genF_i\genK_{\ol{i+1}}\genE_i\genK_{\ol{i+1}}+\genF_i\genK_{\ol{i+1}}^{2}\genE_i\\&\quad+\genE_i\genK_{\ol{i+1}}^{2}\genF_i-v\genE_i\genK_{\ol{i+1}}\genF_i\genK_{\ol{i+1}}-v^{-1}\genK_{\ol{i+1}}\genE_i\genK_{\ol{i+1}}\genF_i+\genK_{\ol{i+1}}\genE_i\genF_i\genK_{\ol{i+1}}\\
 &=(v-v^{-1})^{2}\genK_{\ol{i+1}}^{2}\genF_i\genE_i-(v-v^{-1})^{2}\genK_{\ol{i+1}}\genF_i\genE_i\genK_{\ol{i+1}}+\genE_{\ol{i}}\genF_{\ol{i}}+\genF_{\ol{i}}\genE_{\ol{i}}\\
 &=\frac{\genK_i\genK_{i+1}-\genK_i^{-1}\genK_{i+1}^{-1}}{v-v^{-1}}+(v-v^{-1})^{2}\genK_{\ol{i+1}}^{2}\genF_i\genE_i\\
 &\quad-(v-v^{-1})^{2}\genK_{\ol{i+1}}\genF_i\genE_i\genK_{\ol{i+1}}+(v-v^{-1})\genK_{\ol{i+1}}\genK_{\ol{i}}\\
 &=T_i(\frac{\genK_i\genK_{i+1}-\genK_i^{-1}\genK_{i+1}^{-1}}{v-v^{-1}}+(v-v^{-1})\genK_{\ol{i}}\genK_{\ol{i+1}}).
\end{align*}	
	
 	For (QQ5), we only give the proof of the forth from the bottom and the others are analogous to it.  
 	When {$Y=T_i$}, we have
\begin{align*}
   T_i(\genE_{\ol{i}})T_i(\genE_{\ol{i}})
 &=(-\genK_{\ol{i+1}}\genF_i\genK_i +v\genF_i\genK_{\ol{i+1}}\genK_i)(-\genK_{\ol{i+1}}\genF_i\genK_i+v\genF_i\genK_{\ol{i+1}}\genK_i)\\
&=(v-v^{-1})^{2}\genF_i\genK_{\ol{i+1}}^{2}\genF_i\genK_i^{2}
+(1-v^{2})\genF_{\ol{i}}\genK_{\ol{i+1}}\genF_i\genK_{{i+1}}\genK_i^{2}\\&\quad
+(1-v^{2})\genF_i\genK_{\ol{i+1}}\genF_{\ol{i}}\genK_{{i+1}}\genK_i^{2}
+v^{2}\genF_{\ol{i}}^{2}\genK_{{i+1}}^{2}\genK_i^{2}\\
&=v^{2}\frac{v-v^{-1}}{v+v^{-1}}\genF_i^{2}\genK_{{i+1}}^{2}\genK_i^{2}
-v^{-2}\frac{v-v^{-1}}{v+v^{-1}}\genF_i^{2}\genK_{{i+1}}^{-2}\genK_i^{2}
+v^{2}\frac{v-v^{-1}}{v+v^{-1}}\genF_{{i}}^{2}\genK_{{i+1}}^{2}\genK_i^{2}\\&\quad
+v^{2}(1-v^{2})\genF_{{i}}^{2}\genK_{{i+1}}^{2}\genK_i^{2}
-v^{2}(1-v^{2})\frac{v-v^{-1}}{v+v^{-1}}\genF_{{i}}^{2}\genK_{{i+1}}^{2}\genK_i^{2}\\
&=-v^{-2}\frac{ {v} - {v}^{-1} }{{v} + {v}^{-1}}\genF_i^2\genK_i^2\genK_{i+1}^{-2}\\
&=-\frac{ {v} - {v}^{-1} }{{v} + {v}^{-1}}T_i(\genE_i)^{2},
\end{align*}
\begin{align*}
     T_i({\genE}_i)T_ i({\genE}_{i+1})-T_ i({\genE}_{i+1}) T_i (\genE_i)
 	&=(-\genF_i\genK_i\genK_{i+1}^{-1})(-\genE_i\genE_{i+1}+v^{-1}\genE_{i+1}\genE_i)\\
 	&\quad-v(-\genE_i\genE_{i+1}+v^{-1}\genE_{i+1}\genE_i)(-\genF_i\genK_i\genK_{i+1}^{-1})\\
    &=v\genF_i\genE_i\genE_{i+1} \genK_i\genK_{i+1}^{-1}-\genE_{i+1}\genF_i\genE_i\genK_i\genK_{i+1}^{-1}\\&\quad-v\genE_i\genF_i\genE_{i+1}\genK_i\genK_{i+1}^{-1}+\genE_{i+1}\genE_i\genF_i\genK_i\genK_{i+1}^{-1}\\
 	&=-v\frac{\genK_i\genK_{i+1}^{-1}-\genK_i^{-1}\genK_{i+1}}{v-v^{-1}}\genE_{i+1}\genK_i\genK_{i+1}^{-1}
  \\&\qquad
  +\genE_{i+1}\frac{\genK_i\genK_{i+1}^{-1}-\genK_i^{-1}\genK_{i+1}}{v-v^{-1}}\genK_i\genK_{i+1}^{-1}\\
 	&=-v\frac{\genK_i\genK_{i+1}^{-1}-\genK_i^{-1}\genK_{i+1}}{v-v^{-1}}(\genK_{\ol{i+1}}\genE_{\ol{i+1}}+v\genE_{\ol{i+1}}\genK_{\ol{i+1}})\genK_i\\
 	&\qquad+(\genK_{\ol{i+1}}\genE_{\ol{i+1}}+v\genE_{\ol{i+1}}\genK_{\ol{i+1}})\frac{\genK_i\genK_{i+1}^{-1}-\genK_i^{-1}\genK_{i+1}}{v-v^{-1}}\genK_i\\
 	&=(-\genK_{\ol{i+1}}\genF_i\genK_i+v\genF_i\genK_{\ol{i+1}}\genK_i)(-\genE_i\genE_{\ol{i+1}}+v^{-1}\genE_{\ol{i+1}}\genE_i)\\
 	&\quad+v(-\genE_i\genE_{\ol{i+1}}+v^{-1}\genE_{\ol{i+1}}\genE_i)(-\genK_{\ol{i+1}}\genF_i\genK_i+v\genF_i\genK_{\ol{i+1}}\genK_i)\\
 	&=T_i(\genE_{\ol{i}})T_i (\genE_{\ol{i+1}})+vT_i (\genE_{\ol{i+1}})T_i (\genE_{\ol{i}}).
\end{align*}
 	
 	For the  relations in (QQ6), we only detailedly verify the third relation.  Then we have
\begin{align*}
    &{T_i(\genE_i)}^{2} T_i(\genE_{\ol{i+1}}) -(v+v^{-1}) T_i(\genE_i) T_i(\genE_{\ol{i+1}}) T_i(\genE_i)+ T_i(\genE_{\ol{i+1}}) {T_i(\genE_i)}^{2}\\
 	&\qquad=(-\genF_i\genK_i\genK_{i+1}^{-1})(-\genF_i\genK_i\genK_{i+1}^{-1})(-\genE_i\genE_{\ol{i+1}}+v^{-1}\genE_{\ol{i+1}}\genE_i)\\
      &\qquad\qquad-(v+v^{-1})(-\genF_i\genK_i\genK_{i+1}^{-1})(-\genE_i\genE_{\ol{i+1}}+v^{-1}\genE_{\ol{i+1}}\genE_i)(-\genF_i\genK_i\genK_{i+1}^{-1})\\
      &\qquad\qquad +(-\genE_i\genE_{\ol{i+1}} +v^{-1}\genE_{\ol{i+1}}\genE_i) (-\genF_i\genK_i\genK_{i+1}^{-1}) (-\genF_i\genK_i\genK_{i+1}^{-1})\\
 	&\qquad=\genF_i(\genE_i\genF_i-\genF_i\genE_i)\genE_{\ol{i+1}}-v^{-1}\genE_{\ol{i+1}}\genF_i(\genE_i\genF_i-\genF_i\genE_i)\\
 	&\qquad\qquad-v^{-2}(\genE_i\genF_i-\genF_i\genE_i)\genF_i\genE_{\ol{i+1}} +v^{-3}\genE_{\ol{i+1}} (\genE_i\genF_i-\genF_i\genE_i)\genF_i\\	&\qquad=\genF_i\frac{\genK_i\genK_{i+1}^{-1}-\genK_i^{-1}\genK_{i+1}}{v-v^{-1}}\genE_{\ol{i+1}}-v^{-1}\genE_{\ol{i+1}}\genF_i\frac{\genK_i\genK_{i+1}^{-1}-\genK_i^{-1}\genK_{i+1}}{v-v^{-1}}\\
 	&\qquad\qquad-v^{-2}\frac{\genK_i\genK_{i+1}^{-1}-\genK_i^{-1}\genK_{i+1}} {v-v^{-1}}\genF_i\genE_{\ol{i+1}} +v^{-3}\genE_{\ol{i+1}}\frac{\genK_i\genK_{i+1}^{-1}-\genK_i^{-1}\genK_{i+1}}{v-v^{-1}}\genF_i\\
 	&\qquad=0.
\end{align*}	 
 	 	
 In summary , for each $1\leq i\leq n-1$, $T_i$ maintain the relation (QQ1)-(QQ6), hence $T_i$ is a superalgebra endomorphism.
\end{proof}

\begin{rem}\label{note_i_i1}
	 A direct computation  shows that
\begin{align*}
	&T_iT_j(\genE_i)=\genE_j,\quad T_iT_j(\genF_i)=\genF_j\quad \mbox{if} \enspace |i-j|=1, \\
	&T_iT_{i+1}(\genE_{\ol{i}})=\genE_{\ol{i+1}}
	,\quad 
	T_iT_{i+1}(\genF_{\ol{i}})=\genF_{\ol{i+1}}, \\
&  T_{i+1}T_{i}(\genE_{\ol{i+1}})=  
-\genE_{i+1} \genE_{i} \genK_{\ol{i+2}} \genF_{i+1} \genK_{i+1} 
+v^{-1} \genE_i \genE_{i+1} \genK_{\ol{i+2}} \genF_{i+1} \genK_{i+1} 
+v \genE_{i+1} \genE_i \genF_{i+1} \genK_{\ol{i+2}} \genK_{i+1}   \\
&\qquad \qquad \qquad- \genE_i \genE_{i+1} \genF_{i+1} \genK_{\ol{i+2}} \genK_{i+1}
+ v^{-1} \genK_{\ol{i+2}} \genF_{i+1} \genK_{i+1} \genE_{i+1} \genE_i 
-v^{-2} \genK_{\ol{i+2}} \genF_{i+1} \genK_{i+1} \genE_i \genE_{i+1}\\
&\qquad \qquad \qquad
-\genF_{i+1} \genK_{\ol{i+2}} \genK_{i+1}  \genE_{i+1} \genE_i 
+v^{-1} \genF_{i+1} \genK_{\ol{i+2}} \genK_{i+1}   \genE_i \genE_{i+1},\\ 
&	 T_{i+1}T_{i}(\genF_{\ol{i+1}}) 
=-\genK_{i+1}^{-1} \genE_{i+1} \genK_{\ol{i+2}}\genF_i \genF_{i+1}
+v^{-1}\genK_{\ol{i+2}} \genK_{i+1}^{-1} \genE_{i+1}\genF_i \genF_{i+1}
+v\genK_{i+1}^{-1} \genE_{i+1} \genK_{\ol{i+2}}\genF_{i+1}\genF_i\\
&\qquad \qquad \qquad
-\genK_{\ol{i+2}} \genK_{i+1}^{-1} \genE_{i+1}\genF_{i+1}\genF_i
+v\genF_i \genF_{i+1}\genK_{i+1}^{-1} \genE_{i+1} \genK_{\ol{i+2}}
-v^2 \genF_{i+1}\genF_i\genK_{i+1}^{-1} \genE_{i+1} \genK_{\ol{i+2}}\\
&\qquad \qquad \qquad
-\genF_i \genF_{i+1}\genK_{\ol{i+2}} \genK_{i+1}^{-1} \genE_{i+1}
+v \genF_{i+1}\genF_i\genK_{\ol{i+2}} \genK_{i+1}^{-1} \genE_{i+1}.
\end{align*}	
\end{rem}

\textbf{Proof of the inverse of  $T_i$:}
Obviously, Remark \ref{note_i_i1} shows that  {$T_{i}  \in \fc{B} $} {$(i = 1, \cdots, n-1)$} are automorphisms.
 Also, we have 
\begin{align*}
 &T_{i}^{-1} (\genE_j)=T_{j} (\genE_i)=-\genE_j \genE_i +v^{-1} \genE_i \genE_j, \\
 &T_{i}^{-1} (\genF_j)=T_{j} (\genF_i)=-\genF_i \genF_j +v \genF_j \genF_i \quad \mbox{for } |i-j|=1 ;\\
 &T_{i}^{-1} (\genE_{\ol{i+1}})=T_{i+1} (\genE_{\ol{i}})=-\genE_{i+1}\genE_{\ol{i}}+v^{-1}\genE_{\ol{i}}\genE_{i+1},\\
 &T_{i}^{-1} (\genF_{\ol{i+1}})=T_{i+1} (\genF_{\ol{i}})=-\genF_{\ol{i}}\genF_{i+1} +v \genF_{i+1}\genF_{\ol{i}}.
\end{align*}
And by the formulas in Remark \ref{induction_N}, we see that
\begin{align*}
	T_{i}^{-1}( {\genK}_{\ol{i}}) 
	& = T_{i}^{-1} ( T_{i} ( {\genK}_{\ol{i+1}} ) - (v-v^{-1}) {\genK}_{\ol{i+1}} \genF_i \genE_i  +  (v-v^{-1}) \genF_i \genE_i {\genK}_{\ol{i+1}} ) \\
	& = {\genK}_{\ol{i+1}} -(v-v^{-1}) {\genK}_{\ol{i}} ( - \genE_i \genK_i \genK_{i+1}^{-1} ) ( - \genK_i^{-1} \genK_{i+1} \genF_i)\\
	& \quad
	+ (v-v^{-1}) ( - \genE_i \genK_i \genK_{i+1}^{-1} ) ( - \genK_i^{-1} \genK_{i+1} \genF_i ) {\genK}_{\ol{i}} \\
	&= (v-v^{-1}) \genE_i \genF_i {\genK}_{\ol{i} } - (v-v^{-1}) {\genK}_{\ol{i}} \genE_i \genF_i + {\genK}_{\ol{i+1}},
	\\
	T_{i}^{-1} ( {\genE}_{\ol{i}} )
	& = T_{i}^{-1} ( \genK_{i+1} \genE_{i} {\genK}_{\ol{i+1}} -v\genK_{i+1} {\genK}_{\ol{i+1}} \genE_{i} )\\
	&= \genK_{i} ( - \genK_{i+1} \genK_i^{-1} \genF_i ){\genK}_{\ol{i}}-v\genK_{i} {\genK}_{\ol{i}} ( -  \genK_{i+1} \genK_i^{-1} \genF_i )  \\
	&= -  \genK_{i+1} \genF_i {\genK}_{\ol{i}} + v \genK_{i+1} {\genK}_{\ol{i}} \genF_i,
	\\
	T_{i+1}^{-1} ( {\genE}_{\ol{i}})
	&=T_{i+1}^{-1} ( {\genK}_{\ol{i}} \genE_i \genK_i - v \genE_i {\genK}_{\ol{i}} \genK_i )\\
	&={\genK}_{\ol{i}} ( - \genE_i \genE_{i+1} + v^{-1} \genE_{i+1} \genE_i ) \genK_i \\
	&\quad
	- v ( - \genE_i \genE_{i+1} + v^{-1} \genE_{i+1} \genE_i ) {\genK}_{\ol{i}} \genK_i \\
	&= - ( {\genK}_{\ol{i}} \genE_i - v \genE_i {\genK}_{\ol{i}}) \genE_{i+1} \genK_i
	+ v^{-1} \genE_{i+1}( {\genK}_{\ol{i}} \genE_i-  v  \genE_i  {\genK}_{\ol{i}} ) \genK_i  \\
	& = - ( {\genE}_{\ol{i}} \genK_i^{-1} ) \genE_{i+1} \genK_i  + v^{-1} \genE_{i+1} ( {\genE}_{\ol{i}}  \genK_i^{-1}) \genK_i  \\
	&= - {\genE}_{\ol{i}} \genE_{i+1} + v^{-1} \genE_{i+1} {\genE}_{\ol{i}}.
\end{align*}
 Therefore together with the anti-involution $\Omega$ given in \eqref{omega},  Theorem \ref{action_uvqn} can be proved completely.

\section{The root vectors of $\Uvqn$}\label{root_1}

Similar to \cite[Proposition 9.1.2]{CP}, to construct a PBW-type basis of quantum queer superalgebra $\Uvqn$, we need to define analogues of root vectors associated with the roots of $\qn$. 

Recall the root lattice  {$\mathrm{Q}  = \bigoplus\limits_{j=1}^{n-1} \ZZ \alpha_{j} $} 
and {$\alpha_{j} = \bs{\ep}_j - \bs{\ep}_{j+1} $}.
Assume $\gamma = \sum \limits_{i=1}^{n} k_i \bs{\ep}_i \in \mathrm{Q}$,
 we set $\genK_{\gamma} =\prod\limits_{i=1}^{n}\genK_i^{k_i}$. Then we have the following in $\Uvqn$.
 
\begin{prop}\label{prop5}
For any ${\gamma} = \sum \limits_{i=1}^{n} k_i \bs{\ep}_i  \in \mathrm{Q}$,
and  {$w \in \fS_n$},
we have $T_w(\genK_{\gamma})=\genK_{w({\gamma})}$. 
\end{prop}
\begin{proof}
Assume  {$w = s_{j_1}  s_{j_2} \cdots  s_{j_r}  $} be a reduced expression of {$w$}.
By  Theorem \ref{action_uvqn},  for any {$h \in [1, n]$},
 we have {$T_{j}( \genK_{h} ) = \genK_{s_{j}(h)}$},
and  {$T_{w}( \genK_{h} ) =  T_{j_1}  T_{j_2} \cdots  T_{j_r}  ( \genK_{h} ) =  \genK_{w(h)}$}.
As a consquence, 
\begin{align*}
&T_w(\genK_{\gamma})
=T_w(\prod_{i=1}^{n}\genK_i^{k_i})
=\prod_{i=1}^{n}T_w(\genK_i^{k_i})
=\prod_{i=1}^{n}{(T_w(\genK_i))}^{k_i}
=\prod_{i=1}^{n} \genK_{w(i)}^{k_i}
.
\end{align*}
On the other hand,  \eqref{action_dual_ep} implies
\begin{align*}
&w({\gamma}) 
= w( \sum \limits_{i=1}^{n} k_i \bs{\ep}_i ) 
= \sum \limits_{i=1}^{n} k_i w( \bs{\ep}_i )
= \sum \limits_{i=1}^{n} k_i   \bs{\ep}_{w(i )},
\end{align*}
which means
\begin{align*}
&\genK_{w({\gamma})}=\prod_{i=1}^{n}\genK_{w(i)}^{k_i} = T_w(\genK_{\gamma}) .
\end{align*}
\end{proof}

Similar to \cite[Definition 8.1.4]{CP}, we give the following definition.

\begin{defn}\label{qrv}
For any expression of $w_0 = w^{\bs{\gamma}} = s_{i_1}s_{i_2} \cdots s_{i_N}$  and assume  {$1 \le t \le N$}, 
we define some special elements in $\Uvqn$ as 
\begin{equation}\label{qroot}
\begin{aligned}
	{\genE}^{\bs{\gamma}}_{t}=T_{i_1}T_{i_2}\cdots T_{{i_t}-1}(\genE_{i_t}),\quad
	{\genF}^{\bs{\gamma}}_{t}=T_{i_1}T_{i_2}\cdots T_{{i_t}-1}(\genF_{i_t}), \\	
     	\ol{\genE}^{\bs{\gamma}}_{t}=T_{i_1}T_{i_2}\cdots T_{{i_t}-1}(\genE_{\ol{i_t}}),\quad
     	\ol{\genF}^{\bs{\gamma}}_{t}=T_{i_1}T_{i_2}\cdots T_{{i_t}-1}(\genF_{\ol{i_t}}).
\end{aligned}
\end{equation}
The elements {$ {\genE}^{\bs{\gamma}}_{t}$}'s (resp.  ${\genF}^{\bs{\gamma}}_{t}$'s) are called even positive (resp. negative) root vectors,
while 
the elements {$ \ol{\genE}^{\bs{\gamma}}_{t}$}'s (resp.  $\ol{\genF}^{\bs{\gamma}}_{t}$'s) are called odd positive (resp. negative) root vectors,
\end{defn}

 We denote the set of  positive and negative root vectors by
\begin{align*}
\Psi^+_{\bs{\gamma}} = \{ {\genXE}^{\bs{\gamma}}_{t}, {\ol{\genE}}^{\bs{\gamma}}_{t} \where 1 \le t \le N \},\qquad 
\Psi^-_{\bs{\gamma}} =  \{  {\genF}^{\bs{\gamma}}_{t}, {\ol{\genF}}^{\bs{\gamma}}_{t}  \where 1 \le t \le N \} .
\end{align*}
For any ${\beta}_{\bs{\gamma},t }=\alpha_{i,j} \in \Phi^+ \cup \Phi^-$  ({$1\le i <  j \le n$}),  we can denote
\begin{align*}
  &{\genXE}^{\bs{\gamma}}_{i,j} = {\genXE}^{\bs{\gamma}}_{t},  \quad
 \ol{{\genE}}^{\bs{\gamma}}_{i,j} = \ol{\genXE}^{\bs{\gamma}}_{t},
\\
& {\genXF}^{\bs{\gamma}}_{j,i} = {\genF}^{\bs{\gamma}}_{t},\quad
\ol{{\genXF}}^{\bs{\gamma}}_{j,i} =  	\ol{\genF}^{\bs{\gamma}}_{t}.
\end{align*}

\begin{exam}
Fix $n=3$ and we will find the root vectors for ${\boldsymbol U}_{\!{v}}(\mathfrak{\lcase{q}}_{3})$ 
using two different expressions of the longest element  $w_0$ of {$\fS_3$}.
\begin{enumerate}
\item 
If we choose {$\bs{\gamma}$} = (1,2,1) and  $w_0 = w^{\bs{\gamma}} = s_1 s_2 s_1$,
the positive roots are 
\begin{align*}
\beta_{\bs{\gamma}, 1} = \alpha_1 = \alpha_{1,2}; \quad
\beta_{\bs{\gamma}, 2} = s_1(\alpha_2) = \alpha_{1,3}; \quad
\beta_{\bs{\gamma}, 3} = s_1 s_2 (\alpha_1) = \alpha_{2,3},
\end{align*} 
while the corresponding  positive root vectors are
\begin{align*}
&\genXE^{\bs{\gamma}}_{1,2} = \genE_1 ,\quad
\genXE^{\bs{\gamma}}_{1,3} = T_1 (\genE_2) = -\genE_1 \genE_2 +v^{-1} \genE_2 \genE_1 , \quad 
\genXE^{\bs{\gamma}}_{2,3} = T_1 T_2 (\genE_1) = \genE_2 ,\\
&
\ol{\genXE}^{\bs{\gamma}}_{1,2} = \genE_{\ol{1}} ,\quad 
\ol{\genXE}^{\bs{\gamma}}_{1,3} =T_1 (\genE_{\ol{2}}) = -\genE_1 \genE_{\ol{2}} +v^{-1} \genE_{\ol{2}} \genE_1 , \quad
\ol{\genXE}^{\bs{\gamma}}_{2,3} = T_1 T_2 (\genE_{\ol{1}}) = \genE_{\ol{2}}.
\end{align*} 
\item 
If we choose  {$\bs{\gamma}$} = (2,1,2) and  $w_0 = w^{\bs{\gamma}} = s_2 s_1 s_2$, then
\begin{align*}
\beta_{\bs{\gamma}, 1}  = \alpha_2 = \alpha_{2,3}; \quad
\beta_{\bs{\gamma}, 2} = s_2(\alpha_1) = \alpha_{1,3}; \quad
\beta_{\bs{\gamma}, 3} = s_2 s_1 (\alpha_2) = \alpha_{1,2},
\end{align*} 
and the corresponding positive root vectors are
\begin{align*}
&\genXE^{\bs{\gamma}}_{2,3} =\genE_2 ,\quad 
\genXE^{\bs{\gamma}}_{1,3} =T_2 (\genE_1) = -\genE_2 \genE_1 +v^{-1} \genE_1 \genE_2 , \quad 
\genXE^{\bs{\gamma}}_{1,2} =T_2 T_1 (\genE_2) = \genE_1 ,\\
&\ol{\genXE}^{\bs{\gamma}}_{2,3} = \genE_{\ol{2}} ,\quad 
\ol{\genXE}^{\bs{\gamma}}_{1,3} = T_2 (\genE_{\ol{1}}) = -\genE_2 \genE_{\ol{1}} +v^{-1} \genE_{\ol{1}} \genE_2 , \\
& \ol{\genXE}^{\bs{\gamma}}_{1,2} = T_2 T_1 (\genE_{\ol{2}}) =  
-\genE_2 \genE_1 \genK_{\ol{3}} \genF_2 \genK_2 
+v^{-1} \genE_1 \genE_2 \genK_{\ol{3}} \genF_2 \genK_2 
+v \genE_2 \genE_1 \genF_2 \genK_{\ol{3}} \genK_2 
- \genE_1 \genE_2 \genF_2 \genK_{\ol{3}} \genK_2 \\
&\qquad \qquad \quad+ v^{-1} \genK_{\ol{3}} \genF_2 \genK_2 \genE_2 \genE_1 
-v^{-2} \genK_{\ol{3}} \genF_2 \genK_2 \genE_1 \genE_2
-\genF_2 \genK_{\ol{3}} \genK_2  \genE_2 \genE_1 
+v^{-1} \genF_2 \genK_{\ol{3}} \genK_2   \genE_1 \genE_2 .
\end{align*}
\end{enumerate}
\end{exam}

 With the fixed expression $w^{\bs{\sigma}}$, similar calculation to Corollary \ref{classical_2_0},  then we have the following.

\begin{cor}\label{cor_rv}
 Assume that $1\leq i< n$, $i+1<  j< n$. Then we have
\begin{align*}
     &{\genXE}^{\bs{\sigma}}_{i,i+1}=\genE_i,\quad 
      {\overline {\genXE}^{\bs{\sigma}}_{i,i+1}}=\genE_{\overline{i}},\quad
      {\genXF}^{\bs{\sigma}}_{i+1,i}=\genF_i,\quad 
      {\overline {\genXF}^{\bs{\sigma}}_{i+1,i }}=\genF_{\overline{i}},\\
     	&{\genXE}^{\bs{\sigma}}_{i,j}
     	=-{\genXE}^{\bs{\sigma}}_{i,j-1}{\genXE}^{\bs{\sigma}}_{j-1,j}+v^{-1}{\genXE}^{\bs{\sigma}}_{j-1,j}{\genXE}^{\bs{\sigma}}_{i,j-1} 
     	=[ 
     	   \cdots
     	   [ 
     	     [\genE_i, -\genE_{i+1}]_{v^{-1}}, -\genE_{i+2} 
     	              ]_{v^{-1}}, 
     	               \cdots ,-\genE_{j-1} ]_{v^{-1}},\\
     	&{\ol{\genE}}^{\bs{\sigma}}_{i,j}
     	=-{\genXE}^{\bs{\sigma}}_{i,j-1}{\ol {\genXE}^{\bs{\sigma}}_{j-1,j}}+v^{-1}{\ol {\genXE}^{\bs{\sigma}}_{j-1,j}}{\genXE}^{\bs{\sigma}}_{i,j-1}
     	 =[ 
     	 \cdots
     	    [ 
     	          [\genE_i, -\genE_{i+1}]_{v^{-1}}, -\genE_{i+2} 
     	                    ]_{v^{-1}}, 
     	 \cdots ,-\genE_{\ol{j-1}} ]_{v^{-1}},\\
     	 &{\genXF}^{\bs{\sigma}}_{j,i}
     	 =-{\genXF}^{\bs{\sigma}}_{j,j-1}{\genXF}^{\bs{\sigma}}_{j-1,i}+v{\genXF}^{\bs{\sigma}}_{j-1,i}{\genXF}^{\bs{\sigma}}_{j,j-1}
     	 =[-\genF_{j-1}, 
     	       [-\genF_{j-2}, 
     	          \cdots 
     	           [-\genF_{i+1}, \genF_i]_v 
     	                 ]_v 
     	                   \cdots
     	                                      ]_v,\\
     	  &{\genXF}^{\bs{\sigma}}_{j,i}
     	  =-{\ol {\genXF}^{\bs{\sigma}}_{j,j-1}}{\genXF}^{\bs{\sigma}}_{j-1,i}+v{\genXF}^{\bs{\sigma}}_{j-1,i}{\ol {\genXF}^{\bs{\sigma}}_{j,j-1}}
     	 =[-\genF_{\ol{j-1}}, 
     	      [-\genF_{j-2}, 
     	        \cdots 
     	          [-\genF_{i+1}, \genF_i]_v 
     	                    ]_v 
     	                       \cdots
     	                             ]_v,
\end{align*}    
where  $[a,b]_{x}=ab - xba$,  for any $a,b \in \Uvqn , x \in \Qv$.	
\end{cor}

Similar to the case for $\Uqn$ in Section \ref{root_2}, 
assume $w^{\bs{\varsigma}}$ and $w^{\bs{\tau}}$ are two reduced expressions of $w_0$,
and $w^{\bs{\tau}}$  could be obtained by applying one of the equations in \eqref{sys}, 
we will  find the relations between $\Psi^+_{\bs{\varsigma}} $ and $\Psi^+_{\bs{\tau}}$.

\begin{enumerate}
\item[Case 1:]
Assume that $w^{\bs{\varsigma}} =w{'}s_{{i}} s_{{j}}  w{''}$, $w^{\bs{\tau}} =w{'}s_{{j}} s_{{i}}  w{''}$, 
   with $l(w{'}) = h$ and {$ | {i}  -  {j} | > 1 $}. 
In this case, we have 
\begin{align*}
   &{\genE}_{{t}}^{\bs{\varsigma}}
   ={\genE}_{{t}}^{\bs{\tau}}, \quad 
   \ol{\genE}_{{t}}^{{\bs{\varsigma}}}
   =\ol{\genE}_{{t}}^{\bs{\tau}}, \quad
   {\genF}_{{t}}^{\bs{\varsigma}}
   ={\genF}_{{t}}^{\bs{\tau}}, \quad 
   \ol{\genF}_{{t}}^{\bs{\varsigma}}
   =\ol{\genF}_{{t}}^{\bs{\tau}} \enspace 
   \mbox{ for all  }\enspace t\ne  h+1, h+2 ;\\
   &{\genE}_{{h+1}}^{\bs{\varsigma}}
   =T_{w{'}}(\genE_{{i}})
   =T_{w{'}} T_{j} (\genE_{{i}})
   ={\genE}_{{h+2}}^{\bs{\tau}}, \quad
   {\genE}_{{h+2}}^{\bs{\varsigma}}
   =T_{w{'}}T_{{i}}(\genE_{j})
   =T_{w{'}}(\genE_{j})
   ={\genE}_{{h+1}}^{\bs{\tau}},\\
   &\ol{\genE}_{{h+1} }^{\bs{\varsigma}}
   =T_{w{'}}(\genE_{\ol{{{i}}}})
   =T_{w{'}} T_{j} (\genE_{\ol{i}})
   =\ol{\genE}_{{h+2}}^{\bs{\tau}}\quad
   \ol{\genE}_{{h+2} }^{\bs{\varsigma}}
   =T_{w{'}}T_{{i}}(\genE_{\ol{j}})
   =T_{w{'}}(\genE_{\ol{j}})
   =\ol{\genE}_{{h+1}}^{\bs{\tau}};\\   
   &{\genF}_{{h+1} }^{\bs{\varsigma}}
   =T_{w{'}}(\genF_{{i}})
   =T_{w{'}} T_{j} (\genF_{{i}})
   ={\genF}_{{h+2}}^{\bs{\tau}}, \quad
   {\genF}_{{h+2} }^{\bs{\varsigma}}
   =T_{w{'}}T_{{i}}(\genF_{j})
   =T_{w{'}}(\genF_{j})
   ={\genF}_{{h+1}}^{\bs{\tau}}, \\
   &\ol{\genF}_{{h+1} }^{\bs{\varsigma}}
   =T_{w{'}}(\genF_{\ol{i}})
   =T_{w{'}} T_{j} (\genF_{\ol{i}})
   =\ol{\genF}_{{h+2}}^{\bs{\tau}}, \quad
   \ol{\genF}_{{h+2} }^{\bs{\varsigma}}
   =T_{w{'}}T_{{i}}(\genF_{\ol{j}})
   =T_{w{'}}(\genF_{\ol{j}})
   =\ol{\genF}_{{h+1}}^{\bs{\tau}}.  
\end{align*}
   

\item[Case 2:]
Assume that $w^{\bs{\varsigma}} =w{'}s_{{i}} s_{{j}} s_{{i}} w{''}$, $w^{\bs{\tau}} =w{'} s_{{j}} s_{{i}} s_{{j}}w{''}$, 
  with $l(w{'}) = h$ and {$ | {i}  -  {j} | = 1 $}.
With the same discussion as Section \ref{root_2}
and referring to Remark \ref{note_i_i1},
the relationships between the root vectors are 
\begin{align*}
   &{\genE}_{{t} }^{\bs{\varsigma}}
   ={\genE}_{{t} }^{\bs{\tau}}, \quad 
   \ol{\genE}_{{t} }^{\bs{\varsigma}}
   =\ol{\genE}_{{t} }^{\bs{\tau}}, \quad
   {\genF}_{{t} }^{\bs{\varsigma}}
   ={\genF}_{{t}}^{\bs{\tau}}, \quad 
   \ol{\genF}_{{t} }^{\bs{\varsigma}}
   =\ol{\genF}_{{t}}^{\bs{\tau}} \enspace 
   \mbox{if}\enspace t\ne  h+1,  h+2,  h+3; \\
  &{\genE}_{{h+1} }^{\bs{\varsigma}}
  =T_{w{'}}(\genE_{i})
  =T_{w{'}}T_{j}T_{i}(\genE_{j})
  ={\genE}_{{h+3} }^{\bs{\tau}},\quad  
  {\ol{\genE}}_{{h+1}}^{\bs{\varsigma}}
   =T_{w{'}}(\genE_{\ol{i}}),
  \\  
  &{\genF}_{{h+1} }^{\bs{\varsigma}}
  =T_{w{'}}(\genF_{i})
  =T_{w{'}}T_{j}T_{i}(\genF_{j})
  ={\genF}_{{h+3} }^{\bs{\tau}},\quad  
  {\ol{\genF}}_{{h+1}}^{\bs{\varsigma}}
   =T_{w{'}}(\genF_{\ol{i}});\\
 &{\genE}_{{h+2} }^{\bs{\varsigma}}
 =T_{w{'}}T_{i}(\genE_{j})
 =T_{w{'}}(-\genE_{i}\genE_{j}+v^{-1}\genE_{j} \genE_{i})
 =-{\genE}_{{h+3} }^{\bs{\tau}} {\genE}_{{h+1} }^{\bs{\tau}}+v^{-1}{\genE}_{{h+1} }^{\bs{\tau}}{\genE}_{{h+3} }^{\bs{\tau}},\\
&{\ol{\genE}}_{{h+2} }^{\bs{\varsigma}}
   =T_{w{'}}T_{i}(\genE_{\ol{j}})
   =T_{w{'}}(-\genE_{i}\genE_{\ol{j}}
   +v^{-1}\genE_{\ol{j}}\genE_{i})
   =-{\genE}_{{h+3} }^{\bs{\tau}}{\ol{\genE}}_{{h+1}}^{\bs{\tau}}+v^{-1} {\ol{\genE}}_{{h+1}}^{\bs{\tau}}{\genE}_{{h+3} }^{\bs{\tau}}, \\
 &{\genF}_{{h+2} }^{\bs{\varsigma}}
 =T_{w{'}}T_{i}(\genF_{j})
 =T_{w{'}}(-\genF_{j}\genF_{i} +v\genF_{i}\genF_{j} )
 =-{\genF}_{{h+1} }^{\bs{\tau}}{\genF}_{{h+3} }^{\bs{\tau}} +v{\genF}_{{h+3} }^{\bs{\tau}}{\genF}_{{h+1} }^{\bs{\tau}},\\
&{\ol{\genF}}_{{h+2} }^{\bs{\varsigma}}
   =T_{w{'}}T_{i}(\genF_{\ol{j}})
   =T_{w{'}}(-\genF_{\ol{j}}\genF_{i}
   +v\genF_{i}\genF_{\ol{j}})
   =-{\ol{\genF}}_{{h+1}}^{\bs{\tau}}{\genF}_{{h+3} }^{\bs{\tau}}+v {\genF}_{{h+3} }^{\bs{\tau}}{\ol{\genF}}_{{h+1}}^{\bs{\tau}}; \\
  &
  {\genE}_{{h+3} }^{\bs{\varsigma}}
  =T_{w{'}}T_{i}T_{j}(\genE_{i})
  ,\quad  
   {\ol{\genE}}_{{h+3}}^{\bs{\varsigma}}
   =T_{w{'}}T_{i}T_{j}(\genE_{\ol{i}})
   ,\\
   &
   {\genF}_{{h+3} }^{\bs{\varsigma}}
  =T_{w{'}}T_{i}T_{j}(\genF_{i}),\quad  
   {\ol{\genF}}_{{h+3}}^{\bs{\varsigma}}
   =T_{w{'}}T_{i}T_{j}(\genF_{\ol{i}}).
\end{align*}
if $j=i+1$, then by Remark  \ref{note_i_i1}, we have 
\begin{align*}
&{\genE}_{{h+3} }^{\bs{\varsigma}}
  ={\genE}_{{h+1} }^{\bs{\tau}},\quad 
{\ol{\genE}}_{{h+3}}^{\bs{\varsigma}}
   ={\ol{\genE}}_{{h+2} }^{\bs{\tau}},\quad {\genF}_{{h+3} }^{\bs{\varsigma}}
  ={\genF}_{{h+1} }^{\bs{\tau}},\quad  
   {\ol{\genF}}_{{h+3}}^{\bs{\varsigma}}
   ={\ol{\genF}}_{{h+2} }^{\bs{\tau}}.
\end{align*}  
\end{enumerate}

\section{A PBW-type basis of $\Uvqn$}\label{PBW_1}

 In this section, we will construct a PBW-type basis for the quantum queer superalgebra $\Uvqn$.

With the fixed reduced  expression $w^{\bs{\sigma}}$ of $w_0$ in \eqref{w_0_fix}, 
for any $A=(\SEE{a}_{i,j}|\SOE{a}_{i,j})\in M_n (\mathbb{N}|\mathbb{Z}_2)$, 
$\bs{j} =(j_1 , \dots , j_n )\in \ZZ^{n} $,
denote   
\begin{align*}
&{\genK}_{A, \bs{j}} =\prod\limits_{ i=1}^{n}  \genK_i^{j_i} \genK_{\ol {i}}^{a_{i,i}^{\ol{1}}}, \\
&  {\genE}^{\bs{\sigma}}_{A} =\prod_{n\ge i\ge 1} \prod_{n\ge j\ge i+1} ({\genXE}^{\bs{\sigma}}_{i, j})^{\SE{a}_{i, j}} (\ol{\genXE}^{\bs{\sigma}}_{i, j})^{\SO{a}_{i, j}}, \\
& {\genF}^{\bs{\sigma}}_{A} 
   =\prod_{n\ge j\ge 2} \prod_{1\le i\le j-1}({\genXF}^{\bs{\sigma}}_{j, i})^{\SE{a}_{j, i}} (\ol{\genXF}^{\bs{\sigma}}_{j, i})^{\SO{a}_{j, i}}. 
\end{align*}

    let {$\Uv^{0}$} be the $\Qv$-subalgebra of {$\Uvqn$} generated by $\genK_i^{\pm 1}$ and $\genK_{\ol{i}} $ for $1 \le i \le n $, 
   and let {$\Uv^{+}$} (respectively. {$\Uv^{-}$})be the $\Qv$-subalgebra of {$\Uvqn$} generated by $\genE_j$ and $ \genE_{\ol{j}} $ (respectively. $\genF_j, \genF_{\ol{j}} $) for $1 \le j \le n-1 $. 
   Referring to \cite[Theorem 2.3]{DJ},  there is a linear space isomorphism
   $$
   \Uvqn \cong \Uv^{-} \otimes \Uv^{0} \otimes \Uv^{+}.
   $$

Recall the root vectors defined in \cite{DW}(by replacing quantum parameter $q$ to $v$):
for $1\leq i\leq n-1$,  define
$$
X_{i,i+1}=\genE_i,\quad X_{i+1,i}=\genF_i, \quad  \ol{X}_{i,i+1}=\genE_{\ol{i}},\quad  \ol{X}_{i+1,i}=\genF_{\ol{i}};
$$
for $|j-i|>1$,  define
\begin{equation}\label{q-root}
\aligned
	X_{i,j}&:=\left\{
\begin{array}{ll}
X_{i,k}X_{k,j} - {v} X_{k,j}X_{i,k},&\text{ if }i<j,\\
X_{i,k} X_{k,j} - {v}^{-1} X_{k,j}X_{i,k},&\text{ if }i>j,
\end{array}
\right.\\
	\ol{X}_{i,j} &:=\left\{
\begin{array}{ll}
X_{i,k}  \ol{X}_{k,j}-{v}  \ol{X}_{k,j}X_{i,k},&\text{ if }i<j,\\
 \ol{X}_{i,k} X_{k,j}-{v}^{-1}X_{k,j}  \ol{X}_{i,k},&\text{ if }i>j,
\end{array}
\right.
\endaligned
\end{equation}
where $k$ is strictly between $i$ and $j$.
Referring to Corollary \ref{cor_rv}, 
it is seen that $X_{i,j}$, $\ol{X}_{i,j}$ differs ${\genXE}^{\bs{\sigma}}_{i, j}$, $ \ol{\genXE}^{\bs{\sigma}}_{i, j}$  with $v$ and {${v}^{-1}$}.
Hence, similar to \cite[Proposition 5.8]{DW}, we have the following conclusion. 
\begin{prop}\label{prop6}
   With the fixed expression $w^{\bs{\sigma}}$, we have 
\begin{enumerate}
\item
   	 the set
$
   \{	 {\genE}^{\bs{\sigma}}_{A} =\prod_{n\ge i\ge 1} \prod_{n\ge j\ge i+1} ({\genXE}^{\bs{\sigma}}_{i, j})^{\SE{a}_{i, j}} (\ol{\genXE}^{\bs{\sigma}}_{i, j})^{\SO{a}_{i, j}}
      \where  A \in \MNZN(n)\} 
$
   	forms a $\Qv$-basis of {$\Uv^{+}$};
\item
     the set
$
 \{   {\genF}^{\bs{\sigma}}_{A} 
   =\prod_{n\ge j\ge 2} \prod_{1\le i\le j-1}({\genXF}^{\bs{\sigma}}_{j, i})^{\SE{a}_{j, i}} (\ol{\genXF}^{\bs{\sigma}}_{j, i})^{\SO{a}_{j, i}}
     \where  A \in \MNZN(n)\} 
$
   	forms a $\Qv$-basis of {$\Uv^{-}$};
\item
        the set
   	 $ \{  {\genK}_{A, \bs{j}} =\prod\limits_{ i=1}^{n}  \genK_i^{j_i} \genK_{\ol {i}}^{a_{i,i}^{\ol{1}}} \where \bs{j} =(j_1 , \dots , j_n )\in \ZZ^{n} , A \in \MNZN(n) \}$
   	 forms a $\Qv$-basis of {$\Uv^{0}$};
     \item
   	the set
   	$ \{  \bs{\frm}^{\bs{\sigma}}_{A} = {\genF}^{\bs{\sigma}}_{A} \cdot  {\genK}_{A, \bs{j}}  \cdot {\genE}^{\bs{\sigma}}_{A}  
   	\where  \bs{j} =(j_1 , \dots , j_n )\in \ZZ^{n} , A \in \MNZN(n) \}$
   	   	forms a $\Qv$-basis of {$\Uvqn$}.   
\end{enumerate}
\end{prop}
Referring to Lemma \ref{lemroot},  for $1\le i <  j \le n$, we have 
\begin{equation}\label{vr_d1}
\begin{aligned}
\genE^{\bs{\sigma}}_{\frac{i(i+1)}{2}}=\genXE_{i,i+1}, \qquad
\genF^{\bs{\sigma}}_{\frac{i(i+1)}{2}}=\genXE_{i+1,i}, \\
\ol{\genE}^{\bs{\sigma}}_{\frac{i(i+1)}{2}}=\ol{\genXE}_{i,i+1}, \qquad
\ol{\genF}^{\bs{\sigma}}_{\frac{i(i+1)}{2}}=\ol{\genXE}_{i+1,i};
\end{aligned}
\end{equation}
and when $j>i+1$,  we have
\begin{equation}\label{vr_d2}
\begin{aligned}
\genE^{\bs{\sigma}}_{\frac{(j-2)(j-1)}{2}+i-1}= \genXE_{i,j}, \qquad
\genF^{\bs{\sigma}}_{\frac{(j-2)(j-1)}{2}+i-1}= \genXF_{j,i}, \\
\ol{\genE}^{\bs{\sigma}}_{\frac{(j-2)(j-1)}{2}+i-1}= \ol{\genXE}_{i,j}, \qquad
\ol{\genF}^{\bs{\sigma}}_{\frac{(j-2)(j-1)}{2}+i-1}= \ol{\genXF}_{j,i}.
\end{aligned}
\end{equation}
We define the set for any positive integer $m$, 
\begin{align*}
(\NN | \ZG)^{m} = \{   \bs{k} = (\SE{\bs{k}} | \SO{\bs{k}}) \where \SE{\bs{k}}=(\SE{k}_1, \SE{k}_2, \cdots , \SE{k}_m)\in \NN^{m}, 
\SO{\bs{k}}=(\SO{k}_1, \SO{k}_2, \cdots , \SO{k}_m)\in \ZG^m \}.
\end{align*}
And for any $\bs{t} \in (\NN | \ZG)^{N} $,
we denote the products of all root vectors of $\Uvqn$ by 
\begin{equation}\label{product}
\begin{aligned}
&{\genE}^{\bs{\sigma}, {\bs{t}}}
    =({\genXE}^{\bs{\sigma}}_{N})^{\SE{t}_{N}}  (\ol{\genXE}^{\bs{\sigma}}_{N})^{\SO{t}_{N}}\cdot
    ({\genXE}^{\bs{\sigma}}_{{N-1}})^{\SE{t}_{N-1}} (\ol{\genXE}^{\bs{\sigma}}_{{N-1}})^{\SO{t}_{N-1}}\cdots 
    ({\genXE}^{\bs{\sigma}}_{1})^{\SE{t}_{1}} (\ol{\genXE}^{\bs{\sigma}}_{1})^{\SO{t}_{1}}
    , \\
&{\genF}^{\bs{\sigma}, {\bs{t}}}
    =({\genXF}^{\bs{\sigma}}_{\beta_1})^{\SE{t}_{1}} (\ol{\genF}_{\beta_1})^{\SO{t}_{1}} \cdot
     ({\genXF}^{\bs{\sigma}}_{\beta_{2}})^{\SE{t}_{2}} (\ol{\genF}_{\beta_{2}})^{\SO{t}_{2}} \cdots
     ( {\genXF}^{\bs{\sigma}}_{\beta_N})^{\SE{t}_{N}} (\ol{\genF}_{\beta_N})^{\SO{t}_{N}}.
\end{aligned}
\end{equation}
By equations \eqref{vr_d1},  \eqref{vr_d2} and Proposition \ref{prop6}, 
we have the following results:
\begin{thm}\label{thm_pbwbasis}
   With the fixed expression $w^{\bs{\sigma}}$, we have 
\begin{enumerate}
\item
   	 the set
$
   \{	 {\genE}^{\bs{\sigma}, {\bs{t}}}
      \where  \bs{t} \in (\NN | \ZG)^{N} \} 
$
   	forms a $\Qv$-basis of {$\Uvqn^{+}$};
\item
     the set
$
 \{   {\genF}^{\bs{\sigma}, {\bs{t}}}
     \where  \bs{t} \in (\NN | \ZG)^{N}\} 
$
   	forms a $\Qv$-basis of {$\Uvqn^{-}$};
\item 
the set
$
 \{ {\genF}^{\bs{\sigma}, {\bs{r}}} {\genK}_{\bs{i}, \bs{j} } {\genE}^{\bs{\sigma}, {\bs{t}}}\where \bs{r}, \bs{t} \in (\NN | \ZG)^{N}, \bs{j} =(j_1 , \dots , j_n )\in \ZZ^{n} ,  \bs{i} =(i_1 , \dots , i_n )\in \ZG^{n} \} $ forms a $\Qv$-basis of {$\Uvqn$}.
\\
\end{enumerate}
\end{thm}

 \textbf{Open question:} 
For any reduced decomposition $w^{\bs{\gamma}}$ different from $w^{\bs{\gamma}}$, 
will Theorem \ref{thm_pbwbasis} still be true? 
\\

\textbf{Acknowledgements.}  This work was partially supported by the National Natural Science Foundation of China (Grant Nos. 12371040, 12131018 and 12161141001).
 
\clearpage

\normalem

\end{document}